\documentclass[10pt, a4paper, leqno]{article}
\usepackage[margin=3cm]{geometry}
\usepackage[utf8]{inputenc} 
\usepackage{lmodern}
\usepackage[T1]{fontenc}
\usepackage[english]{babel}

\usepackage[runin]{abstract}
\usepackage{titling}

\usepackage{amsmath}
\usepackage{amsthm}
\usepackage{amsfonts}
\usepackage{amssymb}
\usepackage{rsfso}
\usepackage{mathtools}
\usepackage{clrscode}
\usepackage{enumitem}
\usepackage{mdwlist}

\usepackage[nocompress]{cite}

\abslabeldelim{.}

\newcommand{\keywordsname}{Key words}
\makeatletter
\newcommand{\keywords}[1]{%
\def\thekeywords{#1}%
\begin{@bstr@ctlist}
\hspace*{\abstitleskip}{\abstractnamefont\keywordsname\@bslabeldelim}\abstracttextfont\
#1%
\par\end{@bstr@ctlist}
}
\makeatother

\newcommand{\subjclassname}{Mathematics subject classification}
\makeatletter
\newcommand{\subjclass}[2][2010]{%
\begin{@bstr@ctlist}
\hspace*{\abstitleskip}{\abstractnamefont\subjclassname\ (#1)\@bslabeldelim}\abstracttextfont\
#2%
\par\end{@bstr@ctlist}
}
\makeatother

\makeatletter
\def\and{
	\end{tabular}%
	and%
	\begin{tabular}[t]{c}}%
\makeatother

\makeatletter
\def\thanks#1{
\protected@xdef\@thanks{\@thanks
\protect\footnotetext[\the\c@footnote]{#1}}%
}
\makeatother

\makeatletter
\let\addresses\@empty      
\newcommand{\address}[2][]{\g@addto@macro\addresses{\address{#1}{#2}}}
\newcommand{\curraddr}[2][]{\g@addto@macro\addresses{\curraddr{#1}{#2}}}
\newcommand{\email}[2][]{\g@addto@macro\addresses{\email{#1}{#2}}}
\newcommand{\urladdr}[2][]{\g@addto@macro\addresses{\urladdr{#1}{#2}}}
\def\enddoc@text{
  \ifx\@empty\addresses \else\@setaddresses\fi}
\AtEndDocument{\enddoc@text}
\def\emailaddrname{E-mail address}
\def\@setaddresses{\par
  \nobreak \begingroup
  \interlinepenalty\@M
  \def\address##1##2{\begingroup%
    \par\addvspace\bigskipamount
    \@ifnotempty{##1}{(\ignorespaces##1\unskip) }%
    {\noindent\ignorespaces##2}\par\endgroup}%
  \def\email##1##2{\begingroup
    \@ifnotempty{##2}{\nobreak\noindent\emailaddrname
      \@ifnotempty{##1}{, \ignorespaces##1\unskip}\/:\space
      \ttfamily##2\par}\endgroup}%
  \addresses
  \endgroup
}

\makeatother

\makeatletter
\def\cstar#1{\expandafter\@cstar\csname c@#1\endcsname}
\def\@cstar#1{\ifcase#1\or $\ast$\or $\ast\ast$\or $\ast\ast\ast$\fi}
\AddEnumerateCounter{\cstar}{\@cstar}{$\ast\ast\ast$}
\makeatother

\newlist{conditions}{enumerate}{1}
\newlist{iconditions}{enumerate}{1}
\setlist[conditions]{label=\normalfont(\alph*),ref=\normalfont\alph*}
\setlist[iconditions]{label=\normalfont(\roman*),ref=\normalfont\roman*}

\mathchardef\mhyphen="2D

\newcommand{\CB}{\mathbb{C}}
\newcommand{\F}{\mathbb{F}}
\newcommand{\G}{\mathbb{G}}
\newcommand{\HB}{\mathbb{H}}
\newcommand{\PB}{\mathbb{P}}
\newcommand{\R}{\mathbb{R}}
\newcommand{\SB}{\mathbb{S}}
\newcommand{\T}{\mathbb{T}}
\newcommand{\Z}{\mathbb{Z}}

\newcommand{\C}{\mathcal{C}}
\newcommand{\FC}{\mathcal{F}}
\newcommand{\PC}{\mathcal{P}}
\newcommand{\RC}{\mathcal{R}}
\newcommand{\SC}{\mathcal{S}}
\newcommand{\TC}{\mathcal{T}}
\newcommand{\UC}{\mathcal{U}}

\newcommand{\Halg}{H_{\mathrm{alg}}}
\newcommand{\Mat}{\func{Mat}}
\newcommand{\PT}{\func{PT}}

\newcommand{\codim}{\func{codim}}
\newcommand{\Int}{\func{int}}
\newcommand{\Sing}{\func{Sing}}

\newtheorem{theorem}{Theorem}[section]
\newtheorem{corollary}[theorem]{Corollary}
\newtheorem{proposition}[theorem]{Proposition}
\newtheorem{lemma}[theorem]{Lemma}

\theoremstyle{definition}
\newtheorem{definition}[theorem]{Definition}
\newtheorem{example}[theorem]{Example}
\newtheorem{notation}[theorem]{Notation}
\newtheorem{problem}[theorem]{Problem}
\newtheorem{remark}[theorem]{Remark}

\newcommand{\restr}[2]{#1|_{#2}}

\DeclarePairedDelimiter\abs{\lvert}{\rvert}%
\DeclarePairedDelimiter\norm{\lVert}{\rVert}%

\makeatletter
\let\oldabs\abs
\def\abs{\@ifstar{\oldabs}{\oldabs*}}
\let\oldnorm\norm
\def\norm{\@ifstar{\oldnorm}{\oldnorm*}}
\makeatother

\usepackage[pdftex]{hyperref}

\title{\texorpdfstring{\bf}{}Piecewise-regular maps}
\date{}
\author{Wojciech Kucharz\thanks{The author was partially supported by
National Science Centre (Poland) under grant number 2014/15/B/ST1/00046.}}

\address{Institute of Mathematics\\Faculty of Mathematics and Computer
Science\\Jagiellonian University\\\L{}ojasiewicza 6\\30-348
Krak\'ow\\Poland}
\email{Wojciech.Kucharz@im.uj.edu.pl}

\hypersetup{pdfauthor={\theauthor}, pdftitle={\thetitle}}

\begin{document}
\maketitle
\thispagestyle{empty}

\begin{abstract}
Let $V$, $W$ be real algebraic varieties (that is, up to isomorphism,
real algebraic sets), and $X \subseteq V$ some subset. A map from $X$
into $W$ is said to be \emph{regular} if it can be extended to a regular
map defined on some Zariski locally closed subvariety of $V$ that
contains~$X$. Furthermore, a continuous map $f \colon X \to W$ is said
to be \emph{piecewise-regular} if there exists a stratification $\SC$ of
$V$ such that for every stratum $S \in \SC$ the restriction of $f$ to
each connected component of $X \cap S$ is a regular map. By a
stratification of $V$ we mean a finite collection of pairwise disjoint
Zariski locally closed subvarieties whose union is equal to~$V$.
Assuming that the subset $X$ is compact, we prove that every continuous
map from $X$ into a Grassmann variety or a unit sphere can be
approximated by piecewise-regular maps. As an application, we obtain a
variant of the algebraization theorem for topological vector bundles. If
the variety $V$ is compact and nonsingular, we prove that each
continuous map from $V$ into a unit sphere is homotopic to a
piecewise-regular map of class $\C^k$, where $k$ is an arbitrary
nonnegative integer.
\end{abstract}

\hypersetup{pdfkeywords={\thekeywords}}

\section{Introduction}\label{sec-1}

In this paper, by a \emph{real algebraic variety} we mean a locally
ringed space isomorphic to an algebraic subset of $\R^m$, for some $m$,
endowed with the Zariski topology and the sheaf of real-valued regular
functions (such an object is called an affine real algebraic variety in
\cite{bib6}). The class of real algebraic varieties is identical with
the class of quasi-projective real algebraic varieties, cf.
\cite[Proposition~3.2.10, Theorem~3.4.4]{bib6}. Morphisms of real
algebraic varieties are called \emph{regular maps}. Each real algebraic
variety carries also the Euclidean topology, induced by the standard
metric on $\R$. Unless explicitly stated otherwise, we always refer to
the Euclidean topology.

Topological properties of regular maps and their applications to
algebraization of topological vector bundles were investigated in
numerous papers \citeleft \citen{bib2}\citedash\citen{bib14}\citepunct
\citen{bib16}\citedash\citen{bib25}\citepunct \citen{bib27}\citepunct
\citen{bib31}\citedash\citen{bib35}\citepunct
\citen{bib37}\citedash\citen{bib40}\citepunct
\citen{bib44}\citedash\citen{bib46}\citepunct \citen{bib48}\citepunct
\citen{bib50}\citepunct \citen{bib53}\citepunct \citen{bib61}\citepunct
\citen{bib65}\citepunct \citen{bib68}\citepunct \citen{bib69}\citepunct
\citen{bib73}\citedash\citen{bib77}\citepunct \citen{bib79}\citepunct
\citen{bib80}\citeright. In general, regular maps are too rigid to
reflect adequately topological phenomena. It is therefore desirable to
introduce maps which have many good features of regular maps but are
more flexible.

We first generalize the definition of regular map.

\begin{definition}\label{def-1-1}
Let $V$, $W$ be real algebraic varieties, $X \subseteq V$ some subset,
and $Z$ the Zariski closure of $X$ in $V$.

A map $f \colon X \to W$ is said to be \emph{regular} if there
exist a Zariski open neighborhood $Z_0 \subseteq Z$ of~$X$ and a
regular map $\tilde f \colon Z_0 \to W$ such that $\restr{\tilde
f}{X} = f$.
\end{definition}

The next step requires the concept of stratification. By a
\emph{stratification} of a real algebraic variety~$V$ we mean a
finite collection of pairwise disjoint Zariski locally closed
subvarieties (some possibly empty) whose union is equal to $V$.

\begin{definition}\label{def-1-2}
Let $V$, $W$ be real algebraic varieties, $f \colon X \to W$ a
continuous map defined on some subset $X \subseteq V$, and $\SC$
a stratification of $V$.

The map $f$ is said to be \emph{$\SC$-regular} if for every
stratum $S \in \SC$ the restriction of~$f$ to $X \cap S$ is a
regular map. Also, $f$ is said to be \emph{piecewise
$\SC$-regular} if for every stratum $S \in \SC$ the restriction
of $f$ to each connected component of $X \cap S$ is a regular
map.

Moreover, $f$ is said to be \emph{stratified-regular} (resp.\
\emph{piecewise-regular}) if it is $\TC$-regular (resp.\ piecewise
$\TC$-regular) for some stratification $\TC$ of $V$.
\end{definition}

Essentially, these notions do not depend on the ambient variety
$V$. More precisely, suppose that $V$ is a Zariski locally closed
subvariety of a real algebraic variety~$V'$. The map ${f \colon X
\to W}$ is $\SC$-regular (resp.\ piecewise $\SC$-regular) if and
only if it is $\SC'$-regular (resp.\ piecewise $\SC'$-regular),
where $\SC'$ is the stratification of $V'$ defined by ${\SC' = \SC
\cup \{ \overline V \setminus V \} \cup \{ V' \setminus \overline
V \}}$ with $\overline V$ the Zariski closure of $V$ in $V'$.
Conversely, given a stratification $\PC$ of $V'$, the map $f$ is
$\PC$-regular (resp.\ piecewise $\PC$-regular) if and only if it
is $\PC(V)$-regular (resp.\ piecewise $\PC(V)$-regular), where $\PC(V)$ is
the stratification of $V$ defined by $\PC(V) = \{V\cap P \colon
P \in \PC\}$. Thus, in the definition of stratified-regular
(resp.\ piecewise-regular) it does not matter whether $X$ is
regarded as a subset of $V$ or as a subset of $V'$.

Evidently, each stratified-regular map is piecewise-regular,
whereas the converse is not always true. General properties of
these two classes of maps and relationships between them are
discussed in Section~\ref{sec-2}. Stratified-regular maps and
functions are thoroughly investigated, in a more restrictive
framework, in \cite{bib4, bib28, bib29, bib41, bib42, bib47,
bib49, bib51, bib52, bib55, bib57, bib58, bib59, bib63, bib71,
bib72, bib81}, where they sometimes appear under different names
(cf. Remark~\ref{rem-2-5}). If $X$ is a semialgebraic set, then
each piecewise-regular map defined on~$X$ is semialgebraic.
Theorem~\ref{th-2-9} provides a nontrivial characterization of
piecewise-regular maps among semialgebraic ones.

In this section we concentrate on topological properties of
piecewise-regular maps. With notation as in
Definition~\ref{def-1-2}, we let $\C(X,W)$ denote the space of
all continuous maps from $X$ into $W$, endowed with the
compact-open topology. We say that the map $f$ \emph{can be
approximated by piecewise $\SC$-regular maps} if every
neighborhood of $f$ in $\C(X,W)$ contains a piecewise
$\SC$-regular map. Approximation of~$f$ by maps of other types
(regular, $\SC$-regular, stratified-regular, piecewise-regular,
etc.) is defined in the analogous way.

We pay special attention to maps with values in Grassmannians.
We let $\F$ stand for $\R$, $\CB$ or $\HB$ (the quaternions).
When convenient, $\F$ will be identified with $\R^{d(\F)}$, where
$d(\F) = \dim_{\R}\F$. We will consider only left $\F$-vector
spaces, which plays a role if $\F = \HB$ since the quaternions
are noncommutative. For any integers $r$ and $n$, with $0 \leq r
\leq n$, we denote by $\G_r(\F^n)$ the Grassmann space of
$r$-dimensional $\F$-vector subspaces of $\F^n$. As in
\cite[Sections~3.4 and~13.3]{bib6}, $\G_r(\F^n)$ will be regarded
as a real algebraic variety. The disjoint union
\begin{equation*}
\G(\F^n) = \coprod_{r=0}^n \G_r(\F^n)
\end{equation*}
is also a real algebraic variety.

\begin{theorem}\label{th-1-3}
Let $V$ be a real algebraic variety and let $X \subseteq V$ be
a compact subset. Then, for each positive integer $n$, every
continuous map from $X$ into $\G(\F^n)$ can be approximated by
piecewise-regular maps.
\end{theorem}

Under an additional assumption on $X$, we also have a stronger
result.

\begin{theorem}\label{th-1-4}
Let $V$ be a real algebraic variety and let $X \subseteq V$ be a
compact locally contractible subset. Then there exists a
stratification $\SC$ of $V$ such that, for each positive integer
$n$, every continuous map from $X$ into $\G(\F^n)$ can be
approximated by piecewise $\SC$-regular maps.
\end{theorem}

Virtually all topological spaces one encounters in real algebraic
geometry are locally contractible; for example, any semialgebraic
set is locally contractible, cf. \cite[Theorem~9.3.6]{bib6}.

The proofs of Theorems~\ref{th-1-3} and~\ref{th-1-4}, given in
Section~\ref{sec-4}, are based on some fairly explicit
constructions.

It is well-known that maps with values in $\G(\F^n)$ encode
information on algebraic and topological $\F$-vector bundles, cf.
\cite{bib6, bib36}. This is also the case for
stratified-algebraic $\F$-vector bundles introduced in \cite{bib58} and
further
investigated in \cite{bib54, bib56, bib60, bib63}.
Theorems~\ref{th-1-3} and~\ref{th-1-4} have a bearing on
$\F$-vector bundles as well, which is elaborated upon in
Section~\ref{sec-5}. The main results of Section~\ref{sec-5} are
Theorems~\ref{th-5-10} and~\ref{th-5-11}.

We also have an approximation theorem for maps with values in the
unit $n$-sphere
\begin{equation*}
\SB^n = \{(u_0, \ldots, u_n) \in \R^{n+1} \colon u_0^2 + \cdots +
u_n^2 = 1\}.
\end{equation*}

\begin{theorem}\label{th-1-5}
Let $V$ be a real algebraic variety and let $X \subseteq V$ be a
compact subset. Then, for each positive integer $n$, every
continuous map from $X$ into $\SB^n$ can be approximated by
piecewise-regular maps.
\end{theorem}

Theorem~\ref{th-1-5}, which is proved in Section~\ref{sec-6},
implies that each continuous map from~$X$ into $\SB^n$ is
homotopic to a piecewise-regular map. However, the following
homotopy result requires a different proof.

\begin{theorem}\label{th-1-6}
Let $V$ be a compact nonsingular real algebraic variety, $n$ a
positive integer, and ${f \colon V \to \SB^n}$ a continuous map.
Then there exists a stratification $\SC$ of~$V$ such that, for each nonnegative integer $k$, the map
$f$ is homotopic to a piecewise $\SC$-regular map ${g \colon V \to
\SB^n}$ of class~$\C^k$.
\end{theorem}

We prove Theorem~\ref{th-1-6} in Section~\ref{sec-7}. It turns
out that a suitable stratification $\SC$ of~$V$ is quite simple
and consists of at most 3~strata.

Theorems~\ref{th-1-3}, \ref{th-1-5} and~\ref{th-1-6} are optimal
in the sense explained in the following example.

\begin{example}\label{ex-1-7}
Let $W$ be a compact nonsingular real algebraic variety.
A~cohomology class in $H^q(W; \Z/2)$ is said to be
\emph{algebraic} if the homology class Poincar\'e dual to it can
be represented by a Zariski closed subvariety of $W$ of
codimension $q$, cf. \cite{bib3, bib6}. The set $\Halg^q(W;
\Z/2)$ of all algebraic cohomology classes in $H^q(W; \Z/2)$
forms a subgroup. Obviously, the unique generator~$s_q$ of
$H^q(\SB^q; \Z/2) \cong \Z/2$ is an algebraic cohomology class.

The real algebraic varieties $\G_1(\F^2)$ and $\SB^{d(\F)}$ are
canonically biregularly isomorphic and will be identified. For any
positive integer $m$, let $\T^m = \SB^1 \times \cdots \times
\SB^1$ be the $m$-fold product of $\SB^1$. Clearly,
\begin{equation*}
\T^m \subseteq \R^{2m} = \R^{2m} \times \{0\} \subseteq
\R^{2m+1}.
\end{equation*}
We will regard $\T^m$ as a subset of $\R^{2m+1}$.

Fix an integer $m \geq d(\F)+1$, and let $y_0$ be a point in
$\T^{m-d(\F)}$. Let $\alpha$ be the homology class in
$H_{d(\F)}(\T^m; \Z/2)$ represented by the $\C^{\infty}$
submanifold $N \coloneqq \T^{d(\F)} \times \{y_0\} \subseteq \T^m$. Set
\begin{equation*}
G \coloneqq \{u \in H^{d(\F)}(\T^m; \Z/2) \colon \langle u,
\alpha \rangle = 0 \},
\end{equation*}
where $\langle u, \alpha \rangle$ stands for the Kronecker
product. Let $\pi \colon \T^m = \T^{d(\F)} \times \T^{m-d(\F)} \to
\T^{d(\F)}$ be the canonical projection and let $\tau \colon \T^{d(\F)}
\to \SB^{d(\F)}$ be a $\C^{\infty}$ map of topological degree~$1$. For
the $\C^{\infty}$ map $g \coloneqq \tau \circ \pi \colon \T^m \to
\SB^{d(\F)}$, we have
\begin{equation*}
g^*(s_{d(\F)}) \notin G.
\end{equation*}
Since the normal bundle to $N$ in $\T^m$ is trivial and $N$ is the
boundary of a compact $\C^{\infty}$ manifold with boundary, it follows
from \cite[Proposition~2.5, Theorem~2.6]{bib23} that there exist a
nonsingular real algebraic variety $V$ and a $\C^{\infty}$ diffeomorphism
$\sigma \colon V \to \T^m$ with
\begin{equation*}
\Halg^{d(\F)} (V; \Z/2) \subseteq \sigma^*(G).
\end{equation*}
The map
\begin{equation*}
h \coloneqq g \circ \sigma \colon V \to \G_1(\F^2) = \SB^{d(\F)}
\end{equation*}
is of class $\C^{\infty}$, and
\begin{equation*}
h^*(s_{d(\F)}) \notin \Halg^{d(\F)}(V; \Z/2).
\end{equation*}
By \cite[Propositions~7.2 and~7.7]{bib58}, $h$ is not homotopic to any
stratified-regular map. In particular, $h$ cannot be approximated by
stratified-regular maps, which is of interest in view of
Theorems~\ref{th-1-3} and~\ref{th-1-5}. Also, by
Proposition~\ref{prop-2-13}, $h$ is not homotopic to any
piecewise-regular map of class $\C^{\infty}$; thus the case $k = \infty$
cannot be included in Theorem~\ref{th-1-6}.

Furthermore, according to \cite[Proposition~1.2]{bib15}, $V$ as above
can be assumed to be a Zariski closed subvariety of $\R^{2m+1}$ that is
obtained from $\T^m$ via an arbitrarily small $\C^{\infty}$ isotopy.
\end{example}

From a different viewpoint, strengthening of Theorems~\ref{th-1-3},
\ref{th-1-5} and~\ref{th-1-6} might be possible. Given a compact
nonsingular real algebraic variety $V$ and two positive integers $n$ and
$k$, it remains an open question whether every continuous map from $V$
into $\G(\F^n)$ or $\SB^n$ can be approximated by piecewise-regular maps
of class $\C^k$.

\begin{example}\label{ex-1-8}
Stratified-regular maps often have better approximation and homotopy
properties than regular ones. For instance, if $V$ is a compact
nonsingular real algebraic variety of dimension $n \geq 1$, then every
continuous map from $V$ into $\SB^n$ can be approximated by
stratified-regular maps, cf. \cite[Corollary~1.3]{bib52}. On the other
hand, if $n$ is even, then each regular map from $\T^n$ into $\SB^n$ is
null homotopic, cf. \cite{bib8} or \cite[Theorem~13.5.1]{bib6}.
\end{example}

Piecewise-regular maps are not always more flexible than regular maps.

\begin{example}\label{ex-1-9}
Let $F_r$ be the Fermat curve of degree $r$ in the real projective
plane~$\PB^2(\R)$,
\begin{equation*}
F_r \coloneqq \{(x,y,z) \in \PB^2(\R) \colon x^r + y^r = z^r\}.
\end{equation*}
Clearly, $F_2$ can be identified with $\SB^1$. If $s > r \geq 2$, then
every piecewise-regular map from $F_r$ into~$F_s$ is constant. This
claim holds since, by virtue of the Hurwitz--Riemann theorem
\cite[p.~140]{bib30}, every rational map from $F_r$ into $F_s$ is
constant.
\end{example}

It would be of interest to decide whether or not counterparts of
Theorems~\ref{th-1-3} and~\ref{th-1-5} hold for maps with values in an
arbitrary rational nonsingular real algebraic variety.

We have already indicated how the present paper is organized. It should
be added that Section~\ref{sec-3} contains some preliminary technical
results.

Henceforth, the following notation will be frequently used.

\begin{notation}\label{not-1-10}
For any function $f \colon \Omega \to \R$ defined on some set $\Omega$,
we put
\begin{equation*}
Z(f) \coloneqq \{x \in \Omega \colon f(x)=0\}.
\end{equation*}
\end{notation}

\section{General properties of piecewise-regular maps}\label{sec-2}

We first deal with regular maps in the sense of
Definition~\ref{def-1-1}.

\begin{lemma}\label{lem-2-1}
Let $V$ be a real algebraic variety, $W \subseteq \R^p$ a Zariski closed
subvariety, and
\begin{equation*}
f = (f_1, \ldots, f_p) \colon X \to W \subseteq \R^p
\end{equation*}
a map defined on some subset $X \subseteq V$. Then the following
conditions are equivalent:
\begin{conditions}
\item\label{lem-2-1-a} The map $f$ is regular.
\item\label{lem-2-1-b} Each component function $f_j \colon X \to \R$ is
regular, $j=1,\ldots,p$.
\end{conditions}
\end{lemma}

\begin{proof}
It is clear that (\ref{lem-2-1-a}) implies (\ref{lem-2-1-b}).

Suppose
that (\ref{lem-2-1-b}) holds, and let $Z$ be the Zariski closure of~$X$
in $V$. We can find a Zariski open neighborhood $Z_0 \subseteq Z$ of~$X$
and regular functions $\tilde f_j \colon Z_0 \to \R$ such that
$\restr{\tilde f_j}{X} = f_j$ for $j=1,\ldots,p$. The regular map
$\tilde f =(\tilde f_1, \ldots, \tilde f_p) \colon Z_0 \to \R^p$ is
continuous in the Zariski topology and $\tilde f (X) \subseteq W$. Hence
$\tilde f(Z_0) \subseteq W$, which implies (\ref{lem-2-1-a}).
\end{proof}

Regular functions can be characterized as follows.

\begin{samepage}
\begin{lemma}\label{lem-2-2}
Let $V \subseteq \R^n$ be a Zariski closed subvariety, and $f \colon X
\to \R$ a function defined on some subset $X \subseteq V$. Then the
following conditions are equivalent:
\begin{conditions}
\item\label{lem-2-2-a} The function $f$ is regular.

\item\label{lem-2-2-b} For each point $x \in X$ there exist a Zariski
open neighborhood $V_x \subseteq V$ of $x$ and a regular function $F_x
\colon V_x \to \R$ such that $F_x = f$ on $X \cap V_x$.

\item\label{lem-2-2-c} There exist two regular functions $\varphi, \psi
\colon V \to \R$ such that $X \subseteq V \setminus Z(\psi)$ and $f =
\varphi / \psi$ on $X$.

\item\label{lem-2-2-d} For each point $x \in X$ there exist a Zariski
open neighborhood $U_x \subseteq \R^n$ of $x$ and two polynomial
functions $P_x, Q_x \colon \R^n \to \R$ such that $U_x \subseteq \R^n
\setminus Z(Q_x)$ and $f = P_x / Q_x$ on $X \cap U_x$.\nopagebreak

\item\label{lem-2-2-e} There exist two polynomial functions $P, Q \colon
\R^n \to \R$ such that $X \subseteq \R^n \setminus Z(Q)$ and $f = P/Q$
on $X$.
\end{conditions}
\end{lemma}
\end{samepage}

\begin{proof}
It readily follows that
\begin{equation*}
\textrm{(\ref{lem-2-2-a}) $\Rightarrow$ (\ref{lem-2-2-b}) $\Rightarrow$
(\ref{lem-2-2-d}), (\ref{lem-2-2-e}) $\Rightarrow$ (\ref{lem-2-2-c})
$\Rightarrow$ (\ref{lem-2-2-a}), and (\ref{lem-2-2-e})
$\Rightarrow$ (\ref{lem-2-2-d}).}
\end{equation*}
Suppose that (\ref{lem-2-2-d}) holds. For each point $x \in X$, pick a
polynomial function ${S_x \colon \R^n \to \R}$ with $Z(S_x) = \R^n
\setminus U_x$. Since the Zariski topology on $\R^n$ is Noetherian, we
can find a finite subset $\{ x_1, \ldots, x_r \} \subseteq X$ such that
\begin{equation*}
X \subseteq U \coloneqq U_{x_1} \cup \ldots \cup U_{x_r}.
\end{equation*}
Set $P_i = P_{x_i}$, $Q_i = Q_{x_i}$, $S_i = S_{x_i}$, and
\begin{equation*}
P = \sum_{i=1}^r S_i^2Q_iP_i, \quad Q = \sum_{i=1}^r S_i^2Q_i^2.
\end{equation*}
Then $U = \R^n \setminus Z(Q)$ and $S_i^2Q_if = S_i^2P_i$ on $X$.
Consequently, $f=P/Q$ on $X$, hence (\ref{lem-2-2-e}) holds. The proof
is complete.
\end{proof}

A~\emph{filtration} of a real algebraic variety $V$ is a finite sequence
$\FC = (V_0, V_1, \ldots, V_{m+1})$ of Zariski closed subvarieties of
$V$ satisfying
\begin{equation*}
V = V_0 \supseteq V_1 \supseteq \cdots \supseteq V_{m+1} = \varnothing.
\end{equation*}
We allow $V_i = V_{i+1}$ for some $i$. Note that $\overline{\FC}
\coloneqq \{ V_i \setminus V_{i+1} \colon 0 \leq i \leq m \}$ is a
stratification of $V$.

The following is a generalization of \cite[Proposition~2.2]{bib58}.

\begin{proposition}\label{prop-2-3}
Let $V$, $W$ be real algebraic varieties, and $f \colon X \to W$ a map
defined on some subset $X \subseteq V$. Then the following conditions
are equivalent:
\begin{conditions}
\item\label{prop-2-3-a} There exists a stratification $\SC$ of $V$ such
that for every stratum $S \in \SC$ the restriction of $f$ to $X \cap S$
is a regular map.

\item\label{prop-2-3-b} There exists a filtration $\FC$ of $V$ such that
for every stratum $T \in \overline{\FC}$ the restriction of $f$ to $X
\cap T$ is a regular map.

\item\label{prop-2-3-c}
For every Zariski closed subvariety $Z \subseteq V$, there exists a
Zariski open dense subset $Z^0 \subseteq Z$ such that the restriction of
$f$ to $X \cap Z^0$ is a regular map.
\end{conditions}
In particular, the map $f$ is stratified-regular if and only if it is
continuous and satisfies the equivalent conditions
\normalfont{(\ref{prop-2-3-a})}, \normalfont{(\ref{prop-2-3-b})},
\normalfont{(\ref{prop-2-3-c})}.
\end{proposition}

\begin{proof}
It is clear that (\ref{prop-2-3-b}) implies (\ref{prop-2-3-a}).

Suppose that (\ref{prop-2-3-a}) holds, and let $Z \subseteq V$ be an
irreducible Zariski closed subvariety. We can find a stratum $S \in \SC$
such that the intersection $S \cap Z$ is nonempty and Zariski open
(hence Zariski dense) subset of $Z$. Thus (\ref{prop-2-3-c}) holds for
each irreducible~$Z$. It immediately follows that (\ref{prop-2-3-c})
holds in the general case.

Now suppose that (\ref{prop-2-3-c}) is satisfied. Set $V_0 =V$. Making
use of (\ref{prop-2-3-c}) with $Z=V_0$, we find a Zariski closed nowhere
dense subvariety $V_1 \subseteq V_0$ such that the restriction of $f$ to
$X \cap (V_0 \setminus V_1)$ is a regular map. Note that $\dim V_1 <
\dim V_0$. We repeat this construction with $Z = V_1$ to get $V_2
\subseteq V_1$, and so on. This process terminates after finitely many
steps with $V_{m+1} = \varnothing$, which proves (\ref{prop-2-3-b}).
\end{proof}

Let $V$, $W$ be real algebraic varieties, $X \subseteq V$ some subset, and
$Z$ the Zariski closure of $X$ in~$V$. We say that a map $f \colon X \to
W$ is \emph{rational} if there exists a Zariski open dense subset $Z^0
\subseteq Z$ such that the restriction of $f$ to $X \cap Z^0$ is a
regular map (no condition on the restriction of $f$ to $X \setminus (X
\cap Z^0)$ is imposed).

In view of Proposition~\ref{prop-2-3}, each stratified-regular map is
continuous rational. On the other hand, if the set $\Sing(V)$ of
singular points of $V$ is nonempty, then it can happen that a function
from $V$ into $\R$ is continuous rational but it is not
stratified-regular, cf. \cite[Example~2]{bib42}. However, the following
holds.

\begin{samepage}
\begin{proposition}\label{prop-2-4}
Let $V$, $W$ be real algebraic varieties, and $f \colon U \to W$ a map
defined on an open subset $U \subseteq V \setminus \Sing(V)$. Then the
following conditions are equivalent:\nopagebreak
\begin{conditions}
\item\label{prop-2-4-a} The map $f$ is stratified-regular.

\item\label{prop-2-4-b} The map $f$ is continuous and rational.
\end{conditions}
\end{proposition}
\end{samepage}

\begin{proof}
We may assume that $W \subseteq \R^p$ is a Zariski closed subvariety.
Hence, by Lemma~\ref{lem-2-1}, the proof is reduced to the case $W =
\R$, which follows from \cite[Propostion~4.2]{bib41} (a variant of
\cite[Proposition~8]{bib42}).
\end{proof}

It is worthwhile to make a comment on the terminology used in different
papers.

\begin{remark}\label{rem-2-5}
In view of Proposition~\ref{prop-2-4}, reading papers \cite{bib47,
bib49, bib51, bib52, bib55, bib57} one can substitute everywhere
stratified-regular maps for continuous rational maps. It follows from
Proposition~\ref{prop-2-3} that stratified-regular functions coincide
with continuous \emph{hereditarily rational functions} studied in
\cite{bib41, bib42}. Furthermore, as explained in \cite{bib28, bib58,
bib63}, stratified-regular maps defined on a constructible subset of a
real algebraic variety are identical with \emph{reguluous maps}.
\end{remark}

We now present a counterpart of Proposition~\ref{prop-2-3} for
piecewise-regular maps.

\begin{proposition}\label{prop-2-6}
Let $V$, $W$ be real algebraic varieties, and $f \colon X \to W$ a map
defined on some subset $X \subseteq V$. Then the following conditions
are equivalent:
\begin{conditions}
\item\label{prop-2-6-a} There exists a stratification $\SC$ of $V$ such
that for every stratum $S \in \SC$ the restriction of $f$ to each
connected component of $X \cap S$ is a regular map.

\item\label{prop-2-6-b} There exists a filtration $\FC$ of $V$ such that
for every stratum $T \in \overline{\FC}$ the restriction of $f$ to each
connected component of $X \cap T$ is a regular map.

\item\label{prop-2-6-c} For every Zariski closed subvariety $Z \subseteq
V$ there exists a Zariski open dense subset $Z^0 \subseteq Z$ such that
the restriction of $f$ to each connected component of $X \cap Z^0$ is a
regular map.
\end{conditions}
In particular, the map $f$ is piecewise-regular if and only if it is
continuous and satisfies the equivalent conditions
\normalfont{(\ref{prop-2-6-a})}, \normalfont{(\ref{prop-2-6-b})},
\normalfont{(\ref{prop-2-6-c})}.
\end{proposition}

\begin{proof}
One can repeat the proof of Proposition~\ref{prop-2-3} with only minor
modifications.
\end{proof}

We also have the following characterization of piecewise-regular maps.

\begin{proposition}\label{prop-2-7}
Let $V$, $W$ be real algebraic varieties, and $f \colon X \to W$ a
continuous map defined on some subset $X \subseteq V$. Then the
following conditions are equivalent:
\begin{conditions}
\item\label{prop-2-7-a} The map $f$ is piecewise-regular.

\item\label{prop-2-7-b} There exists a stratification $\SC$ of $V$ such
that for every stratum $S \in \SC$ the restriction $\restr{f}{X \cap S}
\colon X \cap S \to W$ is a piecewise-regular map.
\end{conditions}
\end{proposition}

\begin{proof}
It is clear that (\ref{prop-2-7-a}) implies (\ref{prop-2-7-b}).

Suppose that (\ref{prop-2-7-b}) holds for some stratification $\SC$ of
$V$. For each stratum $S \in \SC$ there exists a stratification $\TC_S$
of $V$ such that for every stratum $T \in \TC_S$ the restriction of $f$
to each connected component of $X \cap S \cap T$ is a regular map. Note
that
\begin{equation*}
\PC \coloneqq \{S \cap T \colon S \in \SC\ \textrm{and} \ T \in \TC_S\}
\end{equation*}
is a stratification of $V$, and the map $f$ is piecewise $\PC$-regular.
Thus (\ref{prop-2-7-b}) implies (\ref{prop-2-7-a}), as required.
\end{proof}

Given a real algebraic variety $V$, a subset $A \subseteq V$ is said to
be a \emph{nonsingular algebraic arc} if its Zariski closure $C$ in $V$
is an algebraic curve (that is, $\dim C =1$), $A \subseteq C \setminus
\Sing(C)$, and $A$ is homeomorphic to $\R$.

\begin{proposition}\label{prop-2-8}
Let $V$, $W$ be real algebraic varieties, and $f \colon X \to W$ a
continuous map defined on a semialgebraic subset $X \subseteq V$. Then
the following conditions are equivalent:
\begin{conditions}
\item\label{prop-2-8-a} The map $f$ is piecewise-regular.

\item\label{prop-2-8-b} There exists a stratification $\SC$ of $V$ such
that for every stratum $S \in \SC$ and every nonsingular algebraic arc
$A$ in $V$, with $A \subseteq X \cap S$, the restriction $\restr{f}{A}$
is a regular map.
\end{conditions}
\end{proposition}

\begin{proof}
We may assume that $W \subseteq \R^p$ is a Zariski closed subvariety.
Hence, by Lemma~\ref{lem-2-1}, the proof is reduced to the case $W=\R$.

It is clear that (\ref{prop-2-8-a}) implies (\ref{prop-2-8-b}).

Suppose that (\ref{prop-2-8-b}) holds for some stratification $\SC$ of
$V$. In view of \cite[Proposition~3.5]{bib41}, for
every stratum $S \in \SC$, the restriction $\restr{f}{X \cap S}$ is a
piecewise-regular map. Consequently, by Proposition~\ref{prop-2-7}, $f$
is a piecewise-regular map. Hence (\ref{prop-2-8-b}) implies
(\ref{prop-2-8-a}), as required.
\end{proof}

Piecewise-regular maps can be characterized among semialgebraic maps as
follows.

\begin{theorem}\label{th-2-9}
Let $V$, $W$ be real algebraic varieties, $X \subseteq V$ a
semialgebraic subset, and $f \colon X \to W$ a continuous semialgebraic
map. Then the following conditions are equivalent:
\begin{conditions}
\item\label{th-2-9-a} The map $f$ is piecewise-regular.

\item\label{th-2-9-b} For every nonsingular algebraic arc $A$ in $V$,
with $A \subseteq X$, the restriction $\restr{f}{A}$ is a
piecewise-regular map.

\item\label{th-2-9-c} For every nonsingular algebraic arc $A$ in $V$,
with $A \subseteq X$, there exists a nonemtpy open subset $A_0 \subseteq
A$ such that the restriction $\restr{f}{A_0}$ is a regular map.
\end{conditions}
\end{theorem}

\begin{proof}
As in the proof of Proposition~\ref{prop-2-8}, we may assume that
$W=\R$.

Evidently, (\ref{th-2-9-a}) implies (\ref{th-2-9-b}), and
(\ref{th-2-9-b}) implies (\ref{th-2-9-c}).

Suppose that (\ref{th-2-9-c}) holds, and let $Z \subseteq V$ be a
Zariski closed subvariety. Let $Y$ be the Zariski closure of $X \cap Z$ in $V$. By Lemma~\ref{lem-2-11} below
(with $X$ replaced by $X \cap Z$), there exists a Zariski open dense
subset $Y^0 \subseteq Y$ such that the restriction of $f$ to each
connected component of $(X \cap Z) \cap Y^0$ is a regular function. Note
that $Z^0 \coloneqq Z \setminus (Y \setminus Y^0)$ is a Zariski open
dense subset of $Z$, and
\begin{equation*}
(X \cap Z) \cap Y^0 = X \cap Z^0.
\end{equation*}
Hence, in view of Proposition~\ref{prop-2-6}, condition (\ref{th-2-9-a})
holds.
\end{proof}

For background on Nash manifolds and Nash functions we refer to
\cite{bib6}. The following variant of \cite[Propositon~2.5]{bib41} will
be useful in the proof of Lemma~\ref{lem-2-11}.

\begin{lemma}\label{lem-2-10}
Let $N \subseteq \R^n$ be a connected Nash submanifold, and $f \colon N
\to \R$ a Nash function. Assume that for every nonsingular algebraic arc
$A$ in $\R^n$, with $A \subseteq N$, there exists a nonempty open subset
$A_0 \subseteq A$ such that the restriction $\restr{f}{A_0}$ is a
regular function. Then $f$ is a rational function.
\end{lemma}

\begin{proof}
Let $V$ be the Zarsiki closure of $N$ in $\R^n$. Note that $V$ is
irreducible. Furthermore, the Zariski closure of the graph of $f$ in $V
\times \R$ is also irreducible. Complexifying these data, we complete
the proof arguing as in \cite[Propostion~2.5]{bib41}.
\end{proof}

In the proof of Theorem~\ref{th-2-9} we used the following.

\begin{lemma}\label{lem-2-11}
Let $V$ be a real algebraic variety, $X \subseteq V$ a semialgebraic
set, and $f \colon X \to \R$ a~semialgebraic function. Let $Y$ be the
Zariski closure of $X$ in $V$. Assume that for every nonsingular
algebraic arc $A$ in $V$, with $A \subseteq X$, there exists a nonempty
open subset $A_0 \subseteq A$ such that the restriction $\restr{f}{A_0}$
is a regular function. Then there exists a Zariski open dense subset
$Y^0 \subseteq Y$ such that the restriction of $f$ to each connected
component of $X \cap Y^0$ is a regular function.
\end{lemma}

\begin{proof}
We may assume that $V \subseteq \R^n$ is a Zariski closed subvariety.
Set $Y_0 = Y \setminus \Sing(Y)$, and let~$X^*$ be the interior of $X
\cap Y_0$ in $Y_0$. Then $X \setminus X^*$ is a semialgebraic subset of
$Y$ whose Zariski closure is Zariski nowhere dense in $Y$, cf.
\cite[Chapter~2]{bib6}. Furthermore, there exists a Zariski closed and
Zariski nowhere dense subvariety $S \subseteq Y$ such that
\begin{equation*}
(X \setminus X^*) \cup \Sing(Y) \subseteq S
\end{equation*}
and the restriction of $f$ to each connected component of $X \setminus
S$ is a Nash function. Since $X \setminus S$ is a semialgebraic set, it
has finitely many connected components. Hence, in view of
Lemma~\ref{lem-2-10}, there exists a Zariski open dense subset $Y^0
\subseteq Y$ which has the required properties.
\end{proof}

We next deal with piecewise-regular maps of class $\C^{\infty}$.
Initially, we consider functions on nonsingular real algebraic arcs.

\begin{lemma}\label{lem-2-12}
Let $C$ be a real algebraic curve, $A \subseteq C \setminus \Sing(V)$ a
nonsingular real algebraic arc, and $f \colon A \to \R$ a
piecewise-regular function of class $\C^{\infty}$. Then $f$ is a regular
function.
\end{lemma}

\begin{proof}
The function $f$ is analytic, being semialgebraic and of class
$\C^{\infty}$, cf. \cite[Proposition~8.1.8]{bib6}.

By definition of piecewise-regular, we can find a Zariski open dense
subset ${C^0 \subseteq C \setminus \Sing(C)}$, a regular function $\varphi
\colon C^0 \to \R$ and a nonempty open subset ${U \subseteq A \cap C^0}$
such that $\restr{f}{U} = \restr{\varphi}{U}$. Regarding $\R$ as a subset
of $\PB^1(\R)$, we get a regular map $\psi \colon C \setminus \Sing(C)
\to \PB^1(\R)$ with ${\restr{\psi}{C^0} = \varphi}$. Hence $f =
\restr{\psi}{A}$ by the identity principle for analytic maps.
Consequently, $f$ is a regular function.
\end{proof}

Lemma~\ref{lem-2-12} can be generalized as follows.

\begin{proposition}\label{prop-2-13}
Let $V$, $W$ be real algebraic varieties, and $f \colon U \to W$ a map
defined on an open subset $U \subseteq V \setminus \Sing(V)$. Then the
following conditions are equivalent:
\begin{conditions}
\item\label{prop-2-13-a} The map $f$ is piecewise-regular and of class
$\C^{\infty}$.

\item\label{prop-2-13-b} The restriction of $f$ to each connected
component of $U$ is a regular map.
\end{conditions}
\end{proposition}

\begin{proof}
As in the proof of Proposition~\ref{prop-2-8}, we may assume that $W = \R$.

Suppose that (\ref{prop-2-13-a}) holds and let $U_0$ be a connected
component of $U$. For each nonsingular algebraic arc $A$ in $V$, with $A
\subseteq U_0$, the restriction $\restr{f}{A}$ is a regular function by
Lemma~\ref{lem-2-12}. Hence, in view of \cite[Theorem~2.4]{bib41},
$\restr{f}{U_0}$ is a rational function. Consequently, $\restr{f}{U_0}$
is a regular function according to \cite[Proposition~2.1]{bib47}.

It is clear that (\ref{prop-2-13-b}) implies (\ref{prop-2-13-a}).
\end{proof}

We also have the following variant of Proposition~\ref{prop-2-13}.

\begin{proposition}\label{prop-2-14}
Let $V$, $W$ be real algebraic varieties, and $f \colon U \to W$ a map
defined on an open subset $U \subseteq V \setminus \Sing(V)$. Then the
following conditions are equivalent:
\begin{conditions}
\item\label{prop-2-14-a} The map $f$ is piecewise-regular, and for every
nonsingular algebraic arc $A$ in $V$, with $A \subseteq U$, the restriction
$\restr{f}{A}$ is of class $\C^{\infty}$.

\item\label{prop-2-14-b} The restriction of $f$ to each connected
component of $U$ is a stratified-regular map.
\end{conditions}
\end{proposition}

\begin{proof}
As in the proof of Proposition~\ref{prop-2-8}, we may assume that $W=\R$.

Suppose that (\ref{prop-2-14-a}) holds, and let $U_0$ be a connected
component of $U$. For each nonsingular algebraic arc $A$ in $V$, with $A
\subseteq U_0$, the restriction $\restr{f}{A}$ is a regular function by
Lemma~\ref{lem-2-12}. Hence, in view of \cite[Theorem~2.4]{bib41},
$\restr{f}{U_0}$ is a rational function. Consequently, $\restr{f}{U_0}$
is a stratified-regular function according to Proposition~\ref{prop-2-4}.

It is clear that (\ref{prop-2-14-b}) implies (\ref{prop-2-14-a}).
\end{proof}

\section{Functions on a simplex}\label{sec-3}

This section is of a technical nature. Our main goal is
Lemma~\ref{lem-3-7}, which is needed in Sections~\ref{sec-4} and
\ref{sec-6}. For the sake of clarity, we begin with some preliminary
facts.

\begin{lemma}\label{lem-3-1}
Let $V$ be a real algebraic variety, $W \subseteq V$ a Zariski closed
subvariety, and $f \colon W \to \R$ a regular function. Then there exists
a regular function $F \colon V \to \R$ such that $\restr{F}{W} = f$.
\end{lemma}

\begin{proof}
By Lemma~\ref{lem-2-1}, there exist regular functions $\varphi, \psi
\colon V \to \R$ such that $W \subseteq V \setminus Z(\psi)$ and $f =
\varphi/\psi$ on $W$. Pick a regular function $\alpha \colon V \to \R$
with $Z(\alpha) = W$. Then the function $F \coloneqq
\frac{\varphi\psi}{\alpha^2 + \psi^2}$ has the required properties.
\end{proof}

For any real algebraic variety $V$, we let $\RC(V)$ denote the ring of
real-valued regular functions on $V$. If $W \subseteq V$ is a Zariski
closed subvariety, then the ideal
\begin{equation*}
I_V(W) = \{f \in \RC(V) \colon \restr{f}{W}=0\}
\end{equation*}
of $\RC(V)$ is called the \emph{ideal of $W$ in $V$}.

\begin{lemma}\label{lem-3-2}
Let $V$ be a real algebraic variety, and $W_1$, $W_2$ Zariski closed
subvarieties of $V$ for which
\begin{equation*}
I_V(W_1 \cap W_2) = I_V(W_1) + I_V(W_2) \ \textrm{\normalfont in} \ \RC(V).
\end{equation*}
Let $f \colon W_1 \cup W_2 \to \R$ be a function such that the
restrictions $\restr{f}{W_1}$ and $\restr{f}{W_2}$ are regular
functions. Then $f$ is a regular function.
\end{lemma}

\begin{proof}
By Lemma~\ref{lem-3-1}, there exists a regular function $g_i \colon V \to
\R$ with $\restr{g_i}{W_i} = \restr{f}{W_i}$ for $i=1,2$. Since $g_1 -
g_2 \in I_V(W_1 \cap W_2)$, we have $g_1 - g_2 = h_1 - h_2$, where $h_i
\in I_V(W_i)$. Hence $g_1 - h_1 = g_2 - h_2$ is a regular function on
$V$ whose restriction to $W_1 \cup W_2$ is equal to $f$. Consequently,
$f$ is a regular function.
\end{proof}

Let $V$ be a nonsingular real algebraic variety, and $D \subseteq V$ a
Zariski closed subvariety. We say that $D$ is a \emph{simple normal
crossing hypersurface} if for each point $p \in D$ there exist a Zariski
open neighborhood $U \subseteq V$ of $p$ and local coordinates $x_1,
\ldots, x_n$ on $U$ (a regular system of parameters at $p \in V$) such
that the intersection of each irreducible component of $D$ with $U$ is
given by the equation $x_i = 0$ for a suitable $i$. In particular, if
$D$ is a simple normal crossing hypersurface, then each irreducible
component of $D$ is nonsingular of codimension $1$.

\begin{lemma}\label{lem-3-3}
Let $V$ be a nonsingular real algebraic variety, $D \subseteq V$ a
simple normal crossing hypersurface, and $f \colon D \to \R$ a function
whose restriction to each irreducible component of $D$ is a regular
function. Then $f$ is a regular function.
\end{lemma}

\begin{proof}
We use induction on the number $k$ of irreducible components of $D$. The
case $k=1$ is obvious. Suppose that $k \geq 2$. Let $D_1$ be an
irreducible component of $D$, and let $D'$ be the union
of the remaining irreducible components. The restriction
$\restr{f}{D_1}$ is a regular function by assumption, whereas the
restriction $\restr{f}{D'}$ is a regular function by the induction
hypothesis.

Pick a point $p \in D_1 \cap D'$. It suffices to find a Zariski open
neighborhood $U \subseteq V$ of $p$ such that $\restr{f}{D \cap U}$ is a
regular function. If $U$ is small enough, there exist local coordinates
$x_1, \ldots, x_n$ on~$U$ such that the ideal $I_U(D_1 \cap U)$ is
generated by $x_1$, and the ideal $I_U(D' \cap U)$ is
generated by the product $x_2 \cdots x_l$ for some $l$ with $2 \leq l
\leq n$. Note that the ideal $I_U(D_1 \cap D' \cap U)$ is generated by
$x_1$ and $x_2 \cdots x_l$; in other words,
\begin{equation*}
I_U(D_1 \cap D' \cap U) = I_U(D_1 \cap U) + I_U(D' \cap U) \ \textrm{in}
\ \RC(U).
\end{equation*}
Hence $\restr{f}{D \cap U}$ is a regular function in view of
Lemma~\ref{lem-3-2}.
\end{proof}

We give next the following variant of Lemma~\ref{lem-3-1}.

\begin{lemma}\label{lem-3-4}
Let $V$ be a real algebraic variety, $Y \subseteq X$ some subsets of
$V$, and $f \colon Y \to \R$ a regular function. Assume that $Y = X \cap
W$, where $W$ is the Zariski closure of $Y$ in $V$. Then there exists a
regular function $F \colon V_0 \to \R$, defined on a Zariski open
neighborhood $V_0$ of $X$ in $V$, such that $\restr{F}{Y} = f$.
\end{lemma}

\begin{proof}
By Lemma~\ref{lem-2-2}, there exist regular functions $\varphi, \psi
\colon V \to \R$ such that $Y \subseteq V \setminus Z(\psi)$ and $f =
\varphi / \psi$ on $Y$. Pick a regular function $\alpha \colon V \to \R$
with $Z(\alpha) = W$. Set $V_0 \coloneqq V \setminus Z(\alpha^2 +
\psi^2)$ and $F \coloneqq \frac{\varphi \psi}{\alpha^2 + \psi^2}$ on
$V_0$. Then $F$ has the required properties.
\end{proof}

\begin{notation}\label{not-3-5}
By a simplex in $\R^m$ we always mean a closed geometric simplex. For
any finite (geometric) simplicial complex $K$ in $\R^m$, we write $|K|$
for the union of all simplices in $K$; thus $|K| \subseteq \R^m$ is a
compact polyhedron. We denote by $K^{(n)}$ the $n$-skeleton of $K$.

If $\Delta \subseteq \R^m$ is a $d$-simplex, then $\dot{\Delta}$ stands
for the simplicial complex which consists of all faces of $\Delta$ of
dimension at most $d-1$. Clearly, $|\dot{\Delta}|$ is the union of all
$(d-1)$-dimensional faces of $\Delta$. The Zariski closure of $\Delta$
in $\R^m$, denoted by $H_{\Delta}$, is an affine subspace of dimension
$d$.
\end{notation}

\begin{lemma}\label{lem-3-6}
Let $\Delta \subseteq \R^m$ be a $d$-simplex and let $f \colon
|\dot{\Delta}| \to \R$ be a function such that the restriction
$\restr{f}{\Gamma} \colon \Gamma \to \R$ is a regular function for every
$(d-1)$-simplex $\Gamma \in \dot{\Delta}$. Then there exists a regular
function $F \colon \Delta \to \R$ with $\restr{F}{|\dot{\Delta}|} = f$.
\end{lemma}

\begin{proof}
Let $\Delta_0, \ldots, \Delta_d$ be all the $(d-1)$-dimensional faces of
$\Delta$. We set $H \coloneqq H_{\Delta}$ and $H_i \coloneqq
H_{\Delta_i}$ for $i=0,\ldots,d$. Obviously, $H_i \subseteq H$ and
$\Delta_i = \Delta \cap H_i$. By Lemma~\ref{lem-3-4}, there exists a
Zariski open neighborhood $\Omega \subseteq H$ of $\Delta$ and a regular
function $G_i \colon \Omega \to \R$ with $\restr{G_i}{\Delta_i} =
\restr{f}{\Delta_i}$ for $i=0,\ldots,d$. In particular, $G_i = G_j$ on
$\Delta_i \cap \Delta_j$ for all $i$, $j$. Since $H_i \cap H_j$ is the
Zariski closure of $\Delta_i \cap \Delta_j$ in~$\R^m$, we get $G_i =
G_j$ on $\Omega \cap H_i \cap H_j$.

Set $D_i \coloneqq \Omega \cap H_i$ and $D \coloneqq D_0 \cup \ldots
\cup D_d$. Then $D$ is a simple normal crossing hypersurface
in~$\Omega$. Define a function $\varphi \colon D \to \R$ by
$\restr{\varphi}{D_i} = \restr{G_i}{D_i}$ for $i=0,\ldots,d$. By
Lemma~\ref{lem-3-3}, $\varphi$ is a~regular function. In view of
Lemma~\ref{lem-3-1}, there exists a regular function $\Phi \colon \Omega
\to \R$ with $\restr{\Phi}{D} = \varphi$. The function $F \coloneqq
\restr{\Phi}{\Delta}$ has the required properties.
\end{proof}

We need the following approximation result for functions defined on a
simplex.

\begin{lemma}\label{lem-3-7}
Let $\Delta \subseteq \R^m$ be a $d$-simplex and let $f \colon \Delta \to
\R$ be a continuous function such that the restriction
$\restr{f}{\Gamma} \colon \Gamma \to \R$ is a regular function for every
$(d-1)$-simplex $\Gamma \in \dot{\Delta}$. Then, for every~$\varepsilon
> 0$, there exists a regular function $g \colon \Delta \to \R$
satisfying
\begin{equation*}
\abs{f(x) - g(x)} < \varepsilon \ \textrm{\normalfont for all} \ x \in \Delta
\end{equation*}
and $\restr{f}{|\dot{\Delta}|} = \restr{g}{|\dot{\Delta}|}$.
\end{lemma}

\begin{proof}
According to Lemma~\ref{lem-3-6}, there exists a regular function $h
\colon \Delta \to \R$ with $\restr{h}{|\dot{\Delta}|} =
\restr{f}{|\dot{\Delta}|}$. By replacing $f$ with $f - h$, the proof is
reduced to the case $\restr{f}{|\dot{\Delta}|} = 0$.

Let $\Delta_0, \ldots, \Delta_d$ be all the $(d-1)$-dimensional faces of
$\Delta$. We set $H \coloneqq H_{\Delta}$ and $H_i \coloneqq
H_{\Delta_i}$ for $i=0, \ldots, d$. The union $D \coloneqq H_0 \cup
\ldots \cup H_d$ is a simple normal crossing hypersurface in $H$. The
function on $\Delta \cup D$ which is equal to $f$ on $\Delta$ and
identically equal to $0$ on $D$ is continuous. Hence, by Tietze's
extension theorem, there exists a continuous function $\varphi \colon H
\to \R$ with $\restr{\varphi}{\Delta} = f$ and $\restr{\varphi}{D} = 0$.

Fix $\varepsilon > 0$. Note that there exists a $\C^{\infty}$ function
$\psi \colon H \to \R$ satisfying
\begin{equation*}
\abs{\varphi(x) - \psi(x)} < \frac{\varepsilon}{2} \ \textrm{for all} \
x \in H
\end{equation*}
and $\restr{\psi}{D} = 0$. Indeed, by the Whitney approximation theorem
\cite[Theorem~10.16]{bib64}, one can find a~$\C^{\infty}$ function
$\lambda \colon H \to \R$ for which
\begin{equation*}
\abs{\varphi(x) - \lambda(x)} < \frac{\varepsilon}{4} \ \textrm{for all}
\ x \in H.
\end{equation*}
Since $\restr{\varphi}{D} = 0$, the set $U \coloneqq \{ x \in H \colon
\abs{\lambda(x)} < \frac{\varepsilon}{4} \}$ is an open neighborhood of
$D$ in $H$. If $\alpha \colon H \to \R$ is a $\C^{\infty}$ function with
$\restr{\alpha}{D} = 1$ and support contained in $U$, then the function
$\psi \coloneqq (1 - \alpha)\lambda$ has the required properties.

Denote by $\C^{\infty}(H)$ the ring of $\C^{\infty}$ real-valued
functions on $H$. One readily sees that the ideal $I \subseteq
\C^{\infty}(H)$ of all $\C^{\infty}$ functions vanishing on $D$ is
generated by polynomial functions, say, $q_1, \ldots, q_r$
(alternatively, one can invoke a much more general result \cite[p.~52,
Proposition~1]{bib78}). Consequently, $\psi$ can be written in the form
\begin{equation*}
\psi = \psi_1 q_1 + \cdots + \psi_r q_r,
\end{equation*}
where the $\psi_k$ are $\C^{\infty}$ functions on $H$. Let
\begin{equation*}
M \coloneqq \sup \{ \abs{q_k(x)} \colon x \in \Delta, \ k=1,\ldots,r \}.
\end{equation*}
By the Weierstrass approximation theorem, there exists a polynomial
function $p_k \colon H \to \R$ satisfying
\begin{equation*}
\abs{\psi_k(x) - p_k(x)} < \frac{\varepsilon}{2rM} \ \textrm{for all} \
x \in \Delta.
\end{equation*}
For $p \coloneqq p_1 q_1 + \cdots + p_r q_r$, we have
\begin{equation*}
\abs{\varphi(x) - p(x)} \leq \abs{\varphi(x) - \psi(x)} + \abs{\psi(x) -
p(x)} < \frac{\varepsilon}{2} + \frac{\varepsilon}{2} = \varepsilon \
\textrm{for all} \ x \in \Delta
\end{equation*}
and $\restr{p}{D} = 0$. We complete the proof setting $g \coloneqq
\restr{p}{\Delta}$.
\end{proof}

\section{Piecewise-regular maps into Grassmannians}\label{sec-4}

The role of Subsections~\ref{sec-4a} and~\ref{sec-4b} is to
review some notation and terminology.

\subsection{Inner product and matrices}\label{sec-4a}

As in Section~\ref{sec-1}, we let $\F$ denote $\R$, $\CB$ or
$\HB$. The $\F$-vector space $\F^n$ is endowed with the standard
inner product
\begin{equation*}
\langle -,- \rangle \colon \F^n \times \F^n \to \F
\end{equation*}
given by
\begin{equation*}
\langle (x_1, \ldots, x_n), (y_1, \ldots, y_n) \rangle =
\sum_{i=1}^n x_i \overline{y_i},
\end{equation*}
where $\overline{y_i}$ stands for the conjugate of $y_i$ in $\F$.

Let $\Mat_{m,n}(\F)$, or simply $\Mat_n(\F)$ if $m=n$, denote the
set of all $m$-by-$n$ matrices with entries in $\F$. For any
matrix $A = [a_{ij}] \in \Mat_{m,n}(\F)$, the corresponding
$\F$-linear transformation $L_A \colon \F^n \to \F^m$ is given by
\begin{equation*}
(x_1, \ldots, x_n) \mapsto (y_1, \ldots, y_m), \ \textrm{where} \
y_i = \sum_{j=1}^n x_j a_{ij} \ \textrm{for} \ i=1,\ldots, m
\end{equation*}
(recall that we always consider left $\F$-vector spaces). We will
identify $A$ with $L_A$ and write 
\begin{equation*}
A(v) = L_A(v) \ \textrm{for} \ v \in \F^n.
\end{equation*}
If $B = [b_{jk}] \in \Mat_{n,r}(\F)$, then we define the
product $AB = [c_{ik}]$ by
\begin{equation*}
c_{ik} = \sum_{j=1}^n b_{jk}a_{ij}.
\end{equation*}
This convention implies that $L_{AB} = L_A \circ L_B$.

We regard $\Mat_{n,r}(\F)$ as a real algebraic variety. If $1
\leq r \leq n$, then the subset 
\begin{equation*}
\Mat_{n,r}^0(\F) \subseteq \Mat_{n,r}(\F)
\end{equation*}
of all matrices with linearly independent
columns is Zariski open. Furthermore, the map
\begin{equation*}
\Mat_{n,r}^0(\F) \to \G_r(\F^n), \quad A \mapsto A(\F^r)
\end{equation*}
is a regular map, as is immediately seen by using the standard
charts on $\G_r(\F^n)$.

\subsection{Vector bundles}\label{sec-4b}

For any topological $\F$-vector bundle $\xi$ on a topological
space $X$, we denote by $E(\xi)$ its total space and by $p(\xi)
\colon E(\xi) \to X$ the bundle projection. The fiber of $\xi$
over a point $x \in X$ is $E(\xi)_x = p(\xi)^{-1}(x)$.

Given a nonnegative integer $n$, we let $\varepsilon_X^n(\F)$
denote the standard product $\F$-vector bundle on~$X$ with total
space $X \times \F^n$. Any morphism $\varphi \colon
\varepsilon_X^r(\F) \to \varepsilon_X^n(\F)$ of topological
$\F$-vector bundles is of the form
\begin{equation*}
\varphi(x,v) = ( x, A_{\varphi}(x)(v) ) \ \textrm{for all} \
(x,v) \in X \times \F^r,
\end{equation*}
where $A_{\varphi} \colon X \to \Mat_{n,r}(\F)$ is a uniquely
determined map, called the \emph{matrix representation} of
$\varphi$. Obviously, $A_{\varphi}$ is a continuous map.

If $\xi$ is a topological $\F$-vector subbundle of
$\varepsilon_X^n(\F)$, then $\varepsilon_X^n(\F) = \xi \oplus
\xi^{\perp}$, where $\xi^{\perp}$ is the orthogonal complement of
$\xi$ with respect to the standard inner product on $\F^n$. The
orthogonal projection $\rho_{\xi} \colon \varepsilon_X^n(\F) \to
\varepsilon_X^n(\F)$ onto $\xi$ is a topological morphism of
$\F$-vector bundles.

We will also consider algebraic vector bundles on a real
algebraic variety $V$. The product $V \times \F^n$ will be
regarded as a real algebraic variety. By an \emph{algebraic
$\F$-vector bundle on $V$} we mean an algebraic $\F$-vector
subbundle of $\varepsilon_V^n(\F)$ for some $n$ (cf.
\cite[Chapters~12 and~13]{bib6} and \cite{bib34, bib35} for
various characterizations of algebraic $\F$-vector bundles).

If $\varphi \colon \varepsilon_V^r(\F) \to \varepsilon_V^n(\F)$
is an algebraic morphism, then the matrix representation
\begin{equation*}
A_{\varphi} \colon V \to \Mat_{n,r}(\F)
\end{equation*}
of $\varphi$ is a regular map.

If $\xi$ is an algebraic $\F$-vector subbundle of
$\varepsilon_V^n(\F)$, then its orthogonal complement
$\xi^{\perp}$ is also an algebraic $\F$-vector subbundle, and the
orthogonal projection $\rho_{\xi} \colon \varepsilon_V^n(\F) \to
\varepsilon_V^n(\F)$ onto $\xi$ is an algebraic morphism.

The tautological $\F$-vector bundle $\gamma_r(\F^n)$ on
$\G_r(\F^n)$ is an algebraic $\F$-vector subbundle of
$\varepsilon_{\G_r(\F)}^n(\F^n)$.

\begin{lemma}\label{lem-4-1}
Let $V$ be a real algebraic variety, and $f \colon X \to
\G_r(\F^n)$ a continuous map defined on some subset $X \subseteq
V$. Then the map $P_f \colon X \to \Mat_n(\F)$, where $P_f(x)
\colon \F^n \to \F^n$ is the orthogonal projection onto $f(x)
\subseteq \F^n$ for all $x \in X$, is continuous. Furthermore, if
the map $f$ is regular, then so is the map $P_f$.
\end{lemma}

\begin{proof}
We regard the pullback $\xi \coloneqq f^*\gamma_r(\F^n)$ as a
topological $\F$-vector subbundle of $\varepsilon_X^n(\F)$, hence
\begin{equation*}
E(\xi)_x = \{x\} \times f(x) \ \textrm{for all} \ x \in X.
\end{equation*}
It follows that $P_f$ is the matrix representation of the
orthogonal projection $\rho_{\xi} \colon \varepsilon_X^n(\F) \to
\varepsilon_X^n(\F)$ onto $\xi$. Consequently, $P_f$ is a
continuous map.

Now, suppose that $f$ is a regular map. It suffices to consider
the case where $X$ is a Zariski locally closed subvariety of $V$.
Then $\xi$ is an algebraic $\F$-vector subbundle of
$\varepsilon_X^n(\F)$, hence the argument above shows that $P_f$
is a regular map.
\end{proof}

\subsection{Maps into Grassmannians}\label{sec-4c}

We can now prove the following variant of Lemma~\ref{lem-3-7}.

\begin{lemma}\label{lem-4-2}
Let $\Delta \subseteq \R^m$ be a $d$-simplex and let $f \colon \Delta \to
\G_r(\F^n)$ be a continuous map such that the restriction
$\restr{f}{\Gamma} \colon \Gamma \to \G_r(\F^n)$ is a regular map for
every $(d-1)$-simplex $\Gamma \in \dot{\Delta}$. Then, for each
neighborhood $\UC \subseteq \C(\Delta, \G_r(\F^n))$ of $f$, there exists a
regular map $g \colon \Delta \to \G_r(\F^n)$ such that $g \in \UC$ and
$\restr{g}{|\dot{\Delta}|} = \restr{f}{|\dot{\Delta}|}$.
\end{lemma}

\begin{proof}
Consider the map $P = P_f \colon \Delta \to \Mat_n(\F)$, where $P(x)
\colon \F^n \to \F^n$ is the orthogonal projection onto $f(x) \subseteq
\F^n$ for all $x \in \Delta$. By Lemma~\ref{lem-4-1}, $P$ is a
continuous map and the restriction $\restr{P}{\Gamma} \colon \Gamma \to
\Mat_n(\F)$ is a regular map for every $(d-1)$-simplex $\Gamma \in
\dot{\Delta}$.

We regard the pullback $\xi \coloneqq f^*\gamma_r(\F^n)$ as a
topological $\F$-vector subbundle of $\varepsilon_{\Delta}^n(\F)$. Since
$\xi$~is topologically trivial, there exists an injective topological
morphism $\varphi \colon \varepsilon_{\Delta}^r(\F) \to
\varepsilon_{\Delta}^n(\F)$ whose image is equal to $\xi$. Let $A =
A_{\varphi} \colon \Delta \to \Mat_{n,r}(\F)$ be the matrix
representation of $\varphi$. Then $A$ is a continuous map and
\begin{equation*}
P(x)A(x) = A(x) \ \textrm{for all} \ x \in \Delta.
\end{equation*}
By the Weierstrass approximation theorem, there exists a regular map $B
\colon \Delta \to \Mat_{n,r}(\F)$ arbitrarily close to $A$. Define $C
\colon \Delta \to \Mat_{n,r}(\F)$ by
\begin{equation*}
C(x) = P(x)B(x) \ \textrm{for all} \ x \in \Delta.
\end{equation*}
Then $C$ is a continuous map, close to $A$, such that the restriction
$\restr{C}{\Gamma} \colon \Gamma \to \Mat_{n,r}(\F)$ is a regular map
for every $(d-1)$-simplex $\Gamma \in \dot{\Delta}$. Hence, according to
Lemma~\ref{lem-3-7}, there exists a regular map $\Phi \colon \Delta \to
\Mat_{n,r}(\F)$, arbitrarily close to $C$, with
$\restr{\Phi}{|\dot{\Delta}|} = \restr{C}{|\dot{\Delta}|}$. In
particular, the $\F$-linear transformation $\Phi(x) \colon \F^r \to
\F^n$ is injective for all $x \in \Delta$. In other words, $\Phi(\Delta)
\subseteq \Mat_{n,r}^0(\F)$. Thus
\begin{equation*}
g \colon \Delta \to \G_r(\F^n), \quad g(x) = \Phi(x)(\F^r)
\end{equation*}
is a well-defined regular map, close to $f$. We may assume that $g \in
\UC$. Furthermore, $\restr{g}{|\dot{\Delta}|} =
\restr{f}{|\dot{\Delta}|}$ since $\restr{\Phi}{|\dot{\Delta}|} =
\restr{C}{|\dot{\Delta}|}$.
\end{proof}

An important consequence of Lemma~\ref{lem-4-2} is the following.

\begin{proposition}\label{prop-4-3}
Let $K$ be a finite simplicial complex in $\R^m$ and let $f \colon |K|
\to \G(\F^n)$ be a continuous map. Then, for each open neighborhood $\UC
\subseteq \C(|K|, \G(\F^n))$ of $f$, there exists a continuous map
$\varphi \colon |K| \to \G(\F^n)$ such that $\varphi \in \UC$ and the
restriction $\restr{\varphi}{\Delta} \colon \Delta \to \G(\F^n)$ is a
regular map for every simplex $\Delta \in K$.
\end{proposition}

\begin{proof}
We use induction on $d=\dim K$. The case $d=0$ is obvious. Suppose now
that $d \geq 1$. By the induction hypothesis, there exists a continuous
map $\psi \colon |K^{(d-1)}| \to \G(\F^n)$, arbitrarily close to
$\restr{f}{|K^{(d-1)}|}$, such that the restriction
$\restr{\psi}{\Gamma} \colon \Gamma \to \G(\F^n)$ is a regular map for
every simplex $\Gamma \in K^{(d-1)}$. It readily follows that $\psi$ has
a continuous extension $\tilde{\psi} \colon |K| \to \G(\F^n)$ that
belongs to~$\UC$. By Lemma~\ref{lem-4-2}, for every $d$-simplex $\Delta
\in K$, there exists a regular map $g_{\Delta} \colon \Delta \to
\G(\F)^n$ close to $\restr{\tilde{\psi}}{\Delta}$ and such that
$\restr{g_{\Delta}}{|\dot{\Delta}|} =
\restr{\tilde{\psi}}{|\dot{\Delta}|}$. Then the map $\varphi \colon |K|
\to \G(\F^n)$, defined by $\restr{\varphi}{|K^{(d-1)}|} = \psi$ and
$\restr{\varphi}{\Delta} = g_{\Delta}$ for every $d$-simplex $\Delta
\in K$, has all the required properties.
\end{proof}

In what follows, we will use stratifications constructed in a fairly
simple way. Any finite simplicial complex $K$ in $\R^m$ gives rise to a
filtration
\begin{equation*}
\FC_K = ( Z_0(K), Z_1(K), \ldots, Z_{m+1}(K) )
\end{equation*}
of $\R^m$, where $Z_0(K) = \R^m$, $Z_{m+1}(K) = \varnothing$, and
$Z_d(K)$ is the union of the $H_{\Delta}$ for all simplices $\Delta \in
K$ of dimension at most $m-d$ with $d = 1, \ldots, m$. Here, as in
Notation~\ref{not-3-5}, $H_{\Delta} \subseteq \R^m$ stands for the
Zariski closure of $\Delta$. Setting
\begin{equation*}
\SC_K \coloneqq \{Z_i(K) \setminus Z_{i+1}(K) \colon i=0,\ldots,m\},
\end{equation*}
we obtain a stratification of $\R^m$. More generally, if $V \subseteq
\R^m$ is a Zariski closed subvariety, then the collection
\begin{equation*}
\SC_K(V) \coloneqq \{V \cap S \colon S \in \SC_K\}
\end{equation*}
is a stratification of $V$, which is said to be \emph{induced by} $K$.
Obviously,
\begin{equation*}
\SC_K(\R^m) = \SC_K.
\end{equation*}

\begin{lemma}\label{lem-4-4}
Let $K$ be a finite simplicial complex in $\R^m$, $V \subseteq \R^m$ a
Zariski closed subvariety, and $X \subseteq |K| \cap V$ an arbitrary
subset. Then, for every stratum $T \in \SC_K(V)$, each connected
component of $X \cap T$ is contained in some simplex $\Delta \in K$.
\end{lemma}

\begin{proof}
It suffices to consider the case $V = \R^m$. If $S \in \SC_K(\R^m) =
\SC_K$, then each connected component of $X \cap S$ is contained in a
connected component of $|K| \cap S$, which in turn is contained in some
simplex $\Delta \in K$ by construction of $\SC_K$.
\end{proof}

We are now ready to prove the first two theorems announced in
Section~\ref{sec-1}.

\begin{proof}[Proof of Theorem~\ref{th-1-3}]
Let $n$ be a positive integer. We may assume that $V \subseteq \R^m$ and
$\G(\F^n) \subseteq \R^N$ are Zariski closed subvarieties.

Consider a continuous map $f \colon X \to \G(\F^n) \subseteq \R^N$. By
Tietze's extension theorem, there exists a continuous map $F \colon \R^m
\to \R^N$ with $\restr{F}{X} = f$. Let $\rho \colon T \to \G(\F^n)$ be a
tubular neighborhood of $\G(\F^n)$ in $\R^N$. Then $U \coloneqq
F^{-1}(T) \subseteq \R^m$ is a neighborhood of $X$, and $\tilde f \colon
U \to \G(\F^n)$, given by $\tilde f(x) = \rho(F(x))$ for $x \in U$, is
a continuous extension of $f$. Since $X$ is a compact set, we can find
a finite simplicial complex $K$ in $\R^m$ with $X \subseteq |K|
\subseteq U$. In view of Proposition~\ref{prop-4-3}, there exists a
continuous map $\varphi \colon |K| \to \G(F^n)$, arbitrarily close to
$\restr{\tilde f}{|K|}$, such that the restriction
$\restr{\varphi}{\Delta} \colon \Delta \to \G(\F^n)$ is a regular map
for every simplex $\Delta \in K$. According to Lemma~\ref{lem-4-4}, the
restriction $g \coloneqq \restr{\varphi}{X} \colon X \to \G(\F^n)$ is a
piecewise $\SC_K(V)$-regular map. The proof is complete since $g$ is
close to $f$.
\end{proof}

\begin{proof}[Proof of Theorem~\ref{th-1-4}]
It can be assumed that $V \subseteq \R^m$ is a Zariski closed
subvariety, hence ${X \subseteq \R^m}$. By Borsuk's theorem
\cite[p.~537]{bib26}, $X$ is a retract of some neighborhood $U \subseteq
\R^m$. We can find a finite simplicial complex $K$ in $\R^m$ with $X
\subseteq |K| \subseteq U$. Thus there exists a retraction $r \colon |K|
\to X$.

We claim that the induced stratification $\SC \coloneqq \SC_K(V)$ of $V$
has all the required properties. Indeed, let $n$ be a positive integer
and let $f \colon X \to \G(\F^n)$ be a continuous map. Then ${f \circ r \colon
|K| \to \G(\F^n)}$ is a continuous extension of $f$. By
Proposition~\ref{prop-4-3}, there exists a continuous map ${\varphi
\colon |K| \to \G(\F^n)}$, arbitrarily close to $f \circ r$, such that
the restriction $\restr{\varphi}{\Delta} \colon \Delta \to \G(\F^n)$ is
a regular map for every simplex $\Delta \in K$. In view of
Lemma~\ref{lem-4-4}, the restriction $g \coloneqq \restr{\varphi}{X} \colon
X \to \G(\F^n)$ is a piecewise $\SC$-regular map. This completes the
proof since $g$ is close to $f$.
\end{proof}

We also have the following variant of Theorem~\ref{th-1-4}.

\begin{theorem}\label{th-4-5}
Let $X_0 \subseteq \R^m$ be a compact subset, $U \subseteq \R^m$ a
neighborhood of $X_0$, $\beta \colon U \to U$ a~homeomorphism, and $K$
is a finite simplicial complex in $\R^m$. Assume that $X_0$ is a retract
of $U$ and
\begin{equation*}
X \coloneqq \beta(X_0) \subseteq |K| \subseteq U.
\end{equation*}
Then, for each positive integer $n$, every continuous map from $X$ into
$\G(\F^n)$ can be approximated by piecewise $\SC_K$-regular maps.
\end{theorem}

\begin{proof}
Let $r_0 \colon U \to X_0$ be a retraction. Then $r_X \colon U \to X$,
given by
\begin{equation*}
r_X(x) = \beta( r_0( \beta^{-1}(x) ) ) \ \textrm{for all} \ x \in U,
\end{equation*}
is a well-defined retraction.

Let $n$ be a positive integer and let $f \colon X \to \G(\F^n)$ be a
continuous map. Since $f \circ r_X \colon U \to \G(\F^n)$ is a
continuous extension of $f$, we complete the proof as in the case of
Theorem~\ref{th-1-4}.
\end{proof}

Theorem~\ref{th-4-5} can be illustrated by revisiting
Example~\ref{ex-1-7}.

\begin{example}\label{ex-4-6}
Fix an integer $m \geq d(\F) + 1$, and consider $\T^m$ as a subset of
$\R^{2m} \times \R = \R^{2m+1}$. Note that $\T^m$ is a retract of $U
\coloneqq (\R \setminus \{0\})^{2m} \times \R$. Let $K$ be a finite
simplicial complex in $\R^{2m+1}$ satisfying
\begin{equation*}
\T^m \subseteq \Int |K| \subseteq |K| \subseteq U.
\end{equation*}
According to Example~\ref{ex-1-7}, we can find a nonsingular Zariski
closed subvariety $V \subseteq \R^{2m+1}$, a $\C^{\infty}$
diffeomorphism $\beta\colon U \to U$, and a $\C^{\infty}$ map
\begin{equation*}
h \colon V \to \G_1(\F^2) = \SB^{d(\F)}
\end{equation*}
such that $V = \beta(\T^m) \subseteq |K|$ and $h$ is not homotopic to
any stratified-regular map.

By Theorem~\ref{th-4-5}, for each positive integer $n$, every continuous
map from $V$ into $\G(\F^n)$ can be approximated by piecewise
$\SC_K$-regular maps. In particular, this is the case for the map $h$.
\end{example}

\section{Piecewise-algebraic vector bundles}\label{sec-5}

Comparison of algebraic and topological vector bundles on a given real
algebraic variety $V$ is a challenging problem, cf. \citeleft
\citen{bib2}\citedash\citen{bib6}\citepunct \citen{bib11}\citepunct
\citen{bib12}\citepunct \citen{bib14}\citepunct \citen{bib21}\citepunct
\citen{bib22}\citepunct \citen{bib25}\citepunct \citen{bib27}\citepunct
\citen{bib31}\citepunct \citen{bib39}\citepunct
\citen{bib44}\citedash\citen{bib46}\citepunct \citen{bib74}\citepunct
\citen{bib76}\citepunct \citen{bib77}\citeright. Unless the variety $V$
is quite exceptional, one can find a topological vector bundle on $V$
that is not topologically isomorphic to any algebraic vector bundle. On
the other hand, in some sense, algebraization of topological vector
bundles is possible, cf. \cite{bib2, bib46, bib66, bib67, bib76, bib77}.
In this section we look at the algebraization problem from a new
perspective. The main results are Theorems~\ref{th-5-10}
and~\ref{th-5-11}, derived from Theorems~\ref{th-1-3} and~\ref{th-1-4},
respectively.

We will use freely notation introduced in Section~\ref{sec-4}. Moreover,
given a topological space $X$, a subspace $A \subseteq X$, and a
topological morphism $\psi \colon \theta \to \omega$ of topological
$\F$-vector bundles on $X$, we let $\psi_A \colon \restr{\theta}{A} \to
\restr{\omega}{A}$ denote the restriction morphism defined by $\psi_A(v)
= \psi(v)$ for all $v \in E(\restr{\theta}{A})$.

We first generalize the definition of algebraic vector bundle.

\begin{definition}\label{def-5-1}
Let $V$ be a real algebraic variety, $X \subseteq V$ some subset, and
$Z$ the Zariski closure of $X$ in $V$.

An \emph{algebraic $\F$-vector bundle} $\xi$ on $X$ is a topological
$\F$-vector subbundle of $\varepsilon_X^n(\F)$, for some~$n$, for which
there exist a Zariski open neighborhood $Z_0 \subseteq Z$ of $X$ and an
algebraic $\F$-vector subbundle~$\tilde{\xi}$ of
$\varepsilon_{Z_0}^n(\F)$ with $\restr{\tilde{\xi}}{X} = \xi$. In that
case, $\xi$ is also said to be an \emph{algebraic $\F$-vector subbundle}
of $\varepsilon_X^n(\F)$. The pair $(Z_0,\tilde{\xi})$ is called an
\emph{algebraic extension} of $\xi$. In particular,
$\varepsilon_X^n(\F)$ is an algebraic $\F$-vector bundle on $X$.

If $\xi$ and $\eta$ are algebraic $\F$-vector bundles on $X$, then an
\emph{algebraic morphism} $\varphi \colon \xi \to \eta$ is a topological
morphism such that there exist algebraic extensions $(Z_0,
\tilde{\xi})$, $(Z_0, \tilde{\eta})$ of $\xi$, $\eta$, respectively,
and an algebraic morphism $\tilde{\varphi} \colon \tilde{\xi} \to
\tilde{\eta}$ with $\tilde{\varphi}_X = \varphi$.
\end{definition}

Algebraic $\F$-vector bundles on $X$, together with algebraic morphisms,
form a category.

\begin{lemma}\label{lem-5-2}
Let $V$ be a real algebraic variety, $X \subseteq V$ some subset, and
$\varphi \colon \xi \to \eta$ a bijective algebraic morphism of
algebraic $\F$-vector bundles on $X$. Then $\varphi$ is an algebraic
isomorphism.
\end{lemma}

\begin{proof}
Let $(Z_0, \tilde{\xi})$, $(Z_0, \tilde{\eta})$, and $\tilde{\varphi}
\colon \tilde{\xi} \to \tilde{\eta}$ be as in Definition~\ref{def-5-1}.
Shrinking $Z_0$ if necessary, we may assume that the algebraic morphism
$\tilde{\varphi}$ is bijective. Thus the proof is reduced to the case
where $X$~is a Zariski locally closed subvariety of $V$.

Our goal is to prove that the inverse map $\varphi^{-1} \colon E(\eta)
\to E(\xi)$ is a regular map. The problem is local for the Zariski
topology on $X$, so it suffices to consider $\xi = \eta =
\varepsilon_X^n(\F)$. Let ${A = A_{\varphi} \colon X \to \Mat_n(\F)}$ be
the matrix representation of $\varphi$. Then $A$ is a regular map and
\begin{equation*}
\varphi^{-1}(x,w) = (x, A(x)^{-1}(w)) \ \textrm{for} \ (x,w) \in X
\times \F^n.
\end{equation*}
Consequently, $\varphi^{-1}$ is a regular map, as required.
\end{proof}

\begin{lemma}\label{lem-5-3}
Let $V$ be a real algebraic variety, $X \subseteq V$ some subset, and
$\xi$ an algebraic $\F$-vector subbundle of $\xi_X^n(\F)$. Then the
orthogonal complement $\xi^{\perp}$ of $\xi$ is also an algebraic
$\F$-vector subbundle of $\varepsilon_X^n(\F)$, and the orthogonal
projection $\rho_{\xi} \colon \varepsilon_X^n(\F) \to
\varepsilon_X^n(\F)$ is an algebraic morphism.
\end{lemma}

\begin{proof}
This is a standard fact if $X$ is a Zariski locally closed subvariety of
$V$. The general case follows immediately.
\end{proof}

We denote by $\gamma(\F^n)$ the algebraic $\F$-vector subbundle of
$\varepsilon_{\G(\F^n)}^n(\F)$ whose restriction to $\G_r(\F^n)$ is
$\gamma_r(\F^n)$ for $0 \leq r \leq n$.

Given a topological space $X$ and a continuous map $f \colon X \to
\G(\F^n)$, we regard the pullback $f^*\gamma(\F^n)$ as a topological
$\F$-vector subbundle of $\varepsilon_X^n(\F)$; thus
\begin{equation*}
E(f^*\gamma(\F^n))_x = \{x\} \times f(x) \ \textrm{for all} \ x \in X.
\end{equation*}
Conversely, if $\xi$ is a topological $\F$-vector subbundle of
$\varepsilon_X^n(\F)$, then the map $f_{\xi}\colon X \to \G(\F^n)$,
defined by
\begin{equation*}
E(\xi)_x = \{x\} \times f_{\xi}(x) \ \textrm{for all} \ x \in X,
\end{equation*}
is continuous and $\xi = f^*_{\xi}\gamma(\F^n)$. We call $f_{\xi}$ the
\emph{classifying map for} $\xi$. Note that the classifying map for
$f^*\gamma(\F^n)$ is $f$.

\begin{lemma}\label{lem-5-4}
Let $V$ be a real algebraic variety, $X \subseteq V$ some subset, and
$\xi$ a topological $\F$-vector subbundle of $\varepsilon_X^n(\F)$. Then
the following conditions are equivalent:
\begin{conditions}
\item\label{lem-5-4-a} $\xi$ is an algebraic $\F$-vector subbundle of
$\varepsilon_X^n(\F)$.

\item\label{lem-5-4-b} The classifying map $f_{\xi} \colon X \to \G(\F)$
for $\xi$ is regular.
\end{conditions}
\end{lemma}

\begin{proof}
This is well-known if $X$ is a Zariski locally closed subvariety of $V$.
The general case follows immediately.
\end{proof}

Next we introduce the central notion of this section.

\begin{definition}\label{def-5-5}
Let $V$ be a real algebraic variety, $X \subseteq V$ some subset, and
$\SC$ a stratification of~$V$.

An \emph{$\SC$-algebraic} (resp.\ a \emph{piecewise $\SC$-algebraic})
\emph{$\F$-vector bundle} $\xi$ on $X$ is a topological $\F$-vector
subbundle of $\varepsilon_X^n(\F)$, for some $n$, such that for every
stratum $S \in \SC$ the restriction $\restr{\xi}{X \cap S}$ is an
algebraic $\F$-vector subbundle of $\varepsilon_{X \cap S}^n(\F)$ (resp.\
for every stratum $S \in \SC$ and each connected component $\Sigma$ of
$X \cap S$ the restriction $\restr{\xi}{\Sigma}$ is an algebraic
$\F$-vector subbundle of $\varepsilon_{\Sigma}^n(\F)$). In that case,
$\xi$ is also said to be an \emph{$\SC$-algebraic} (resp.\ a
\emph{piecewise $\SC$-algebraic}) \emph{$\F$-vector subbundle} of
$\varepsilon_X^n(\F)$.

If $\xi$ and $\eta$ are $\SC$-algebraic (resp.\ piecewise
$\SC$-algebraic) $\F$-vector bundles on $X$, then an
\emph{$\SC$-algebraic} (resp.\ a \emph{piecewise $\SC$-algebraic})
\emph{morphism} $\varphi \colon \xi \to \eta$ is a topological morphism
such that for every stratum $S \in \SC$ the restriction $\varphi_{X \cap
S} \colon \restr{\xi}{X \cap S} \to \restr{\eta}{X \cap S}$ is an
algebraic morphism (resp.\ for every stratum $S \in \SC$ and each
connected component $\Sigma$ of $X \cap S$ the restriction
$\varphi_{\Sigma} \colon \restr{\xi}{\Sigma} \to \restr{\eta}{\Sigma}$
is an algebraic morphism).

A \emph{stratified-algebraic} (resp.\ a \emph{piecewise-algebraic})
$\F$-vector bundle on $X$ is a $\TC$-algebraic (resp.\ a piecewise
$\TC$-algebraic) $\F$-vector bundle on $X$ for some stratification $\TC$
of $V$.

If $\xi$ and $\eta$ are stratified-algebraic (resp.\ piecewise-algebraic)
$\F$-vector bundles on $X$, then a \emph{stratified-algebraic} (resp.\ a
\emph{piecewise-algebraic}) morphism $\varphi \colon \xi \to \eta$ is a
$\TC$-algebraic (resp.\ a piecewise $\TC$-algebraic) morphism for some
stratification $\TC$ of $V$ such that both $\xi$ and $\eta$ are
$\TC$-algebraic (resp.\ piecewise $\TC$-algebraic) $\F$-vector bundles on
$X$.
\end{definition}

It is clear that $\SC$-algebraic (resp.\ piecewise $\SC$-algebraic)
$\F$-vector bundles on $X$, together with $\SC$-algebraic (resp.\
piecewise $\SC$-algebraic) morphisms, form a category. Similarly,
stratified-algebraic (resp.\ piecewise-algebraic) $\F$-vector bundles on
$X$, together with stratified-algebraic (resp.\ piecewise-algebraic)
morphisms, form a category.

In a somewhat less general context, $\SC$-algebraic and
stratified-algebraic $\F$-vector bundles are thoroughly investigated in
\cite{bib54, bib56, bib58, bib60, bib63}.

\begin{proposition}\label{prop-5-6}
Let $V$ be a real algebraic variety, $X \subseteq V$ some subset, $\SC$
a stratification of~$V$, and $\xi$ a topological $\F$-vector subbundle
of $\varepsilon_X^n(\F)$ for some $n$. Then the following conditions are
equivalent:
\begin{conditions}
\item\label{prop-5-6-a} $\xi$ is an $\SC$-algebraic (resp.\ a piecewise
$\SC$-algebraic) $\F$-vector subbundle of $\varepsilon_X^n(\F)$.

\item\label{prop-5-6-b} The classifying map $f_{\xi} \colon X \to
\G(\F^n)$ for $\xi$ is $\SC$-regular (resp.\ piecewise $\SC$-regular).
\end{conditions}
\end{proposition}

\begin{proof}
This follows from Lemma~\ref{lem-5-4}.
\end{proof}

As a direct consequence, we obtain the following.

\begin{corollary}\label{cor-5-7}
Let $V$ be a real algebraic variety, $X \subseteq V$ some subset, and
$\xi$ a topological $\F$-vector subbundle of $\varepsilon_X^n(\F)$ for
some $n$. Then the following conditions are equivalent:
\begin{conditions}
\item\label{cor-5-7-a} $\xi$ is a stratified-algebraic (resp.\ a
piecewise-algebraic) $\F$-vector subbundle of $\varepsilon_X^n(\F)$.

\item\label{cor-5-7-b} The classifying map $f_{\xi} \colon X \to \G(\F^n)$
for $\xi$ is stratified-regular (resp.\ piecewise-regular).
\end{conditions}
\end{corollary}

Vector bundles introduced in Definition~\ref{def-5-5} can be compared
with topological vector bundles.

\begin{proposition}\label{prop-5-8}
Let $V$ be a real algebraic variety, $X \subseteq V$ a compact subset,
and $\SC$ a stratification of $V$. Let $\xi$ and $\eta$ be
$\SC$-algebraic (resp.\ piecewise $\SC$-algebraic) $\F$-vector bundles on
$X$ that are topologically isomorphic. Then $\xi$ and $\eta$ are also
isomorphic in the category of $\SC$-algebraic (resp.\ piecewise
$\SC$-algebraic) $\F$-vector bundles on $X$.
\end{proposition}

\begin{proof}
We consider explicitly only the piecewise $\SC$-algebraic case, the
$\SC$-algebraic one being completely analogous.

Thus, $\xi$ (resp.\ $\eta$) is a piecewise $\SC$-algebraic $\F$-vector
subbundle of $\varepsilon_X^k(\F)$ (resp.\ $\varepsilon_X^l(\F)$) for
some positive integer $k$ (resp.\ $l$).
Since $\varepsilon_X^k(\F) = \xi \oplus \xi^{\perp}$ and
$\varepsilon_X^l(\F) = \eta \oplus \eta^{\perp}$, there exists a
topological morphism $\varphi \colon \varepsilon_X^k(\F) \to
\varepsilon_X^l(\F)$ which transforms $\xi$ onto $\eta$. Let $A =
A_{\varphi} \colon X \to \Mat_{l,k}(\F)$ be the matrix representation of
$\varphi$. By the Weierstrass approximation theorem, we can find a
regular map $B \colon X \to \Mat_{l,k}(\F)$ that is close to $A$. Then
$\psi \colon \varepsilon_X^k(\F) \to \varepsilon_X^l(\F)$, defined by
\begin{equation*}
\psi(x,v) = (x, B(x)(v)) \ \textrm{for} \ (x,v) \in X \times \F^k,
\end{equation*}
is an algebraic morphism.

In view of Lemma~\ref{lem-5-3}, the orthogonal projection $\rho_{\eta}
\colon \varepsilon_X^l(\F) \to \varepsilon_X^l(\F)$ onto $\eta$ is a
piecewise $\SC$-algebraic morphism. Hence $\rho_{\eta} \circ \psi \colon
\varepsilon_X^k(\F) \to \varepsilon_X^l(\F)$ is a piecewise
$\SC$-algebraic morphism which transforms $\xi$ onto $\eta$.
Consequently, the morphism $\sigma \colon \xi \to \eta$ determined by
$\rho_{\eta} \circ \psi$ is bijective and piecewise $\SC$-algebraic. It
follows from Lemma~\ref{lem-5-2} that $\sigma$ is a piecewise
$\SC$-algebraic isomorphism.
\end{proof}

Proposition~\ref{prop-5-8} implies immediately the following.

\begin{corollary}\label{cor-5-9}
Let $V$ be a real algebraic variety, and $X \subseteq V$ a compact
subset. Let $\xi$ and~$\eta$ be stratified-algebraic (resp.\
piecewise-algebraic) $\F$-vector bundles on $X$ that are topologically
isomorphic. Then $\xi$ and $\eta$ are also isomorphic in the category of
stratified-algebraic (resp.\ piecewise-algebraic) $\F$-vector bundles on
$X$.
\end{corollary}

As an application of Theorem~\ref{th-1-3}, we obtain the following
result.

\begin{theorem}\label{th-5-10}
Let $V$ be a real algebraic variety and let $X \subseteq V$ be a compact
subset. Then each topological $\F$-vector bundle on $X$ is topologically
isomorphic to a piecewise-algebraic $\F$-vector bundle on $X$, which is
uniquely determined up to piecewise-algebraic isomorphism.
\end{theorem}

\begin{proof}
Let $\xi$ be a topological $\F$-vector bundle on $X$. Since $X$ is
compact, one can find a positive integer $n$ and a continuous map $f
\colon X \to \G(\F^n)$ such that $\xi$ is topologically isomorphic to
the pullback $f^*\gamma(\F^n)$, cf. \cite[Chapter~3,
Proposition~5.8]{bib36}. According to Theorem~\ref{th-1-3}, $f$ is
homotopic to a piecewise-regular map $g \colon X \to \G(\F^n)$, hence
$\xi$ is topologically isomorphic to the pullback $\eta \coloneqq
g^*\gamma(\F^n)$, cf. \cite[Chapter~3, Theorem~4.7]{bib36}. By
Corollary~\ref{cor-5-7}, $\eta$ is a piecewise-algebraic $\F$-vector
bundle on $X$. The proof is complete in view of Corollary~\ref{cor-5-9}.
\end{proof}

In a similar way, we can derive from Theorem~\ref{th-1-4} the next
result.

\begin{theorem}\label{th-5-11}
Let $V$ be a real algebraic variety and let $X \subseteq V$ be a compact
locally contractible subset. Then there exists a stratification $\SC$ of
$V$ such that each topological $\F$-vector bundle on~$X$ is
topologically isomorphic to a piecewise $\SC$-algebraic $\F$-vector
bundle on $X$, which is uniquely determined up to piecewise
$\SC$-algebraic isomorphism.
\end{theorem}

\begin{proof}
According to Theorem~\ref{th-1-4}, there exists a stratification $\SC$
of $V$ such that, for each positive integer $n$, every continuous map
from $X$ into $\G(\F^n)$ is homotopic to a piecewise $\SC$-regular map.

Let $\xi$ be a topological $\F$-vector bundle on $X$. One can find a
positive integer $n$ and a continuous map $f \colon X \to \G(\F^n)$ such
that $\xi$ is topologically isomorphic to the pullback
$f^*\gamma(\F^n)$. Hence $\xi$ is topologically isomorphic to the
pullback $\eta \coloneqq g^*\gamma(\F^n)$, where $g \colon X \to
\G(\F^n)$ is a piecewise $\SC$-regular map homotopic to $f$. By
Proposition~\ref{prop-5-6}, $\eta$ is a piecewise $\SC$-algebraic
$\F$-vector bundle on~$X$. The proof is complete in view of
Proposition~\ref{prop-5-8}.
\end{proof}

In Theorems~\ref{th-5-10} and~\ref{th-5-11}, piecewise-algebraic and
piecewise $\SC$-algebraic cannot be replaced by stratified-algebraic and
$\SC$-algebraic, respectively.

\begin{example}\label{ex-5-12}
Fix an integer $m \geq d(\F)+1$, and let $V$ and $h \colon V \to
\G_1(\F^2) = \SB^{d(\F)}$ be as in Example~\ref{ex-1-7}. We claim that
the topological $\F$-vector bundle $\xi \coloneqq h^*\gamma_1(\F^2)$ on
$V$ is not topologically isomorphic to any stratified-algebraic
$\F$-vector bundle. Indeed, write $\xi_{\R}$ and $\gamma_1(\F^2)_{\R}$
to indicate that $\xi$ and $\gamma_1(\F^2)$ are regarded as topological
$\R$-vector bundles. Since $w_{d(\F)}(\gamma_1(\F^2)_{\R}) = s_{d(\F)}$,
we get
\begin{equation*}
w_{d(\F)}(\xi_{\R}) = h^*(s_{d(\F)}) \notin \Halg^{d(\F)}(V; \Z/2),
\end{equation*}
where $w_q(-)$ stands for the $q$th Stiefel--Whitney class. The claim
follows in view of \cite[Propositions~7.3 and~7.7]{bib58}.
\end{example}

\section{Piecewise-regular maps into spheres (approximation)}\label{sec-6}

For maps with values in $\SB^n$, we have the following counterpart of
Proposition~\ref{prop-4-3}.

\begin{proposition}\label{prop-6-1}
Let $K$ be a finite simplicial complex in $\R^m$ and let $f \colon |K|
\to \SB^n$ be a continuous map. Assume that $f(\Delta) \neq \SB^n$ for
every simplex $\Delta \in K$. Then, for each open neighborhood $\UC
\subseteq \C(|K|, \SB^n)$ of $f$, there exists a continuous map $\varphi
\colon |K| \to \SB^n$ such that $\varphi \in \UC$ and the restriction
$\restr{\varphi}{\Delta} \colon \Delta \to \SB^n$ is a regular map for
every simplex $\Delta \in K$.
\end{proposition}

\begin{proof}
We use induction on $d = \dim K$. The case $d=0$ is obvious. Suppose now
that $d \geq 1$. By the induction hypothesis, there exists a continuous
map $\psi \colon |K^{(d-1)}| \to \SB^n$, arbitrarily close to
$\restr{f}{|K^{(d-1)}|}$, such that the restriction $\restr{\psi}{\Gamma}
\colon \Gamma \to \SB^n$ is a regular map for every simplex $\Gamma \in
K^{(d-1)}$. It readily follows that $\psi$ can be extended to a
continuous map $\tilde{\psi} \colon |K| \to \SB^n$ that belongs to $\UC$
and satisfies $\tilde{\psi}(\Delta) \neq \SB^n$ for every simplex
$\Delta \in K$. Since $\SB^n$ with one point removed is biregularly
isomorphic to $\R^n$, it follows from Lemma~\ref{lem-3-7} that for every
$d$-simplex $\Delta \in K$, there exists a regular map $g_{\Delta}
\colon \Delta \to \SB^n$ close to $\restr{\tilde{\psi}}{\Delta}$ and
such that $\restr{g_{\Delta}}{|\dot{\Delta}|} =
\restr{\tilde{\psi}}{|\dot{\Delta}|}$. Then the map $\varphi \colon |K|
\to \SB^n$, defined by $\restr{\varphi}{|K^{(d-1)}|} = \psi$ and
$\restr{\varphi}{\Delta} = g_{\Delta}$ for every $d$-simplex
$\Delta \in K$, has all the required properties.
\end{proof}

In what follows we will make use of Lemma~\ref{lem-4-4}.

\begin{proof}[Proof of Theorem~\ref{th-1-5}]
Let $n$ be a positive integer and let $f \colon X \to \SB^n \subseteq
\R^{n+1}$ be a continuous map. We may assume that $V \subseteq \R^m$ is
a Zariski closed subvariety, hence $X \subseteq \R^m$. By Tietze's
extension theorem, there exists a continuous map $F \colon \R^m \to
\R^{n+1}$ with $\restr{F}{X} =f$. Denoting by $\rho \colon \R^{n+1}
\setminus \{0\} \to \SB^n$ the radial projection, we see that $U
\coloneqq F^{-1}(\R^{n+1} \setminus \{0\}) \subseteq \R^m$ is a
neighborhood of $X$, and $\tilde f \colon U \to \SB^n$, given by $\tilde
f(x) = \rho(F(x))$ for $x \in U$, is a continuous extension of $f$.

We can find a finite simplicial complex $K$ in $\R^m$ with $X \subseteq
|K| \subseteq U$. By replacing $K$ with a suitable iterated barycentric
subdivision of $K$, we get $\tilde f(\Delta) \neq \SB^n$ for every
simplex $\Delta \in K$. In view of Proposition~\ref{prop-6-1}, we can
find a continuous map $\varphi \colon |K| \to \SB^n$, arbitrarily close
to $\restr{\tilde f}{|K|}$, such that the restriction $\restr{\varphi}{\Delta}
\colon \Delta \to \SB^n$ is a regular map for every simplex $\Delta \in
K$. According to Lemma~\ref{lem-4-4}, the restriction $g \coloneqq
\restr{\varphi}{X} \colon X \to \SB^n$ is a piecewise $\SC_K(V)$-regular
map. The proof is complete since $g$ is close to $f$.
\end{proof}

It is not clear whether there is a counterpart of Theorem~\ref{th-1-4}
for maps into spheres.

\begin{problem}\label{prob-6-2}
Let $V$ be a real algebraic variety, $X \subseteq V$ a compact locally
contractible subset, and $n$ a positive integer. Does there exist a
stratification $\SC$ of $V$ such that every continuous map from~$X$ into
$\SB^n$ can be approximated by piecewise $\SC$-regular maps?
\end{problem}

In view of Theorem~\ref{th-1-4}, the answer is affirmative if $n=1,2$ or
$4$ since $\SB^{d(\F)} = \G_1(\F^2)$.

\section{Piecewise-regular maps into spheres (homotopy)}\label{sec-7}

The following lemma will be used in the proof of Theorem~\ref{th-1-6}.

\begin{lemma}\label{lem-7-1}
Let $V$ be a compact nonsingular real algebraic variety, $Z \subseteq V$
a nonsingular Zariski closed subvariety with $\codim_V Z \geq 1$, and
$K$ the union of some connected components of $Z$. Then there exists a
closed tubular neighborhood $T \subseteq V$ of $K$ such that $T \cap (Z
\setminus K) = \varnothing$ and the boundary $\partial T$ of $T$ is a
nonsingular Zariski closed subvariety of $V$. Furthermore, for each
nonnegative integer~$l$, there exists a $\C^l$ function $\alpha \colon V
\to \R$ with the following properties:
\begin{iconditions}
\item\label{lem-7-1-i} $Z(\alpha) = K$;
\item\label{lem-7-1-ii} the restrictions $\restr{\alpha}{T}$ and
$\restr{\alpha}{V \setminus \Int T}$ are regular functions.
\end{iconditions}
\end{lemma}

\begin{proof}
Let $T \subseteq V$ be a closed tubular neighborhood of $K$ with $T
\cap (Z \setminus K) = \varnothing$. Note that the homology class in
$H_*(V; \Z/2)$ represented by the $\C^{\infty}$ hypersurface $\partial T
\subseteq V$ is equal to $0$. Hence there exists a $\C^{\infty}$
diffeomorphism $\sigma \colon V \to V$, arbitrarily close in the
$\C^{\infty}$ topology to the identity map of $V$, such that
$\sigma(\partial T)$ is a nonsingular Zariski closed subvariety of $V$
and $\sigma(x) = x$ for all~$x \in Z$, cf. \cite[Theorem~12.4.11]{bib6}.
Replacing $T$ by $\sigma(T)$, we may assume that $\partial T$ is a
nonsingular Zariski closed subvariety of $V$.

Let $f$ and $g$ be real-valued regular functions on $V$ with $Z(f) = Z$
and $Z(g) = \partial T$. If $l$ is a nonnegative integer, then
\begin{equation*}
\alpha(x)=
\begin{cases}
f(x)^2 & \textrm{for} \ x \in T\\
f(x)^2 + g(x)^{2l} & \textrm{for} \ x \in V \setminus T
\end{cases}
\end{equation*}
is a $\C^l$ function that satisfies (\ref{lem-7-1-i}) and
(\ref{lem-7-1-ii}).
\end{proof}

Our proof of Theorem~\ref{th-1-6} depends on the Pontryagin--Thom
construction. Unless explicitly stated otherwise, all $\C^{\infty}$
manifolds will be without boundary. Submanifolds will be closed subsets
of the ambient manifold. The unit $n$-sphere $\SB^n$ will be oriented as
the boundary of the unit $(n+1)$-disc. For any compact $\C^{\infty}$
manifold $X$ there is a canonical one-to-one correspondence
\begin{equation*}
\pi^n(X) \to F^n(X),
\end{equation*}
where $\pi^n(X)$ is the set of all homotopy classes of continuous maps
from $X$ into $\SB^n$, and $F^n(X)$ is the set of framed cobordism
classes of framed submanifolds of $X$ of codimension $n$, cf.
\cite{bib43, bib70} for details. Given a continuous map $f \colon X \to
\SB^n$, we denote by $\PT (f)$ the element of $F^n(X)$ corresponding to
the homotopy class of $f$.

It is convenient to introduce some notation related to this
construction. A framed submanifold of $X$ of codimension $n$ is a pair
$(M,F)$, where $M \subseteq X$ is a $\C^{\infty}$ codimension $n$
submanifold, and $F = (v_1, \ldots, v_n)$ is a $\C^{\infty}$ framing of
the normal bundle to $M$ in $X$. The normal space to $M$ in~$X$ at $x
\in M$ is the quotient $N_xM \coloneqq T_xX /T_xM$ of the tangent
spaces; thus $(v_1(x), \ldots, v_n(x))$ is a basis for $N_xM$.

Given a continuous map $f \colon X \to \SB^n$ and a point $y \in \SB^n$,
suppose that for some open neighborhood $U \subseteq \SB^n$ of $y$ the
restriction map $\restr{f}{f^{-1}(U)} \colon f^{-1}(U) \to U$ is of
class $\C^{\infty}$ and transverse to $y$. Choose a positively oriented
basis $B = (w_1, \ldots, w_n)$ for $T_y\SB^n$. Then $\PT(f) \in F^n(X)$
is represented by the framed submanifold $(f^{-1}(y), F(f,B))$, where
$F(f,B) = (v_1, \ldots, v_n)$ and $(v_1(x), \ldots, v_n(x))$ is
transformed onto $(w_1, \ldots, w_n)$ by the isomorphism $N_xM \to
T_y\SB^n$ induced by the derivative $d_x\varphi \colon T_xX \to
T_y\SB^n$ for every $x \in f^{-1}(y)$.

Let $\varphi \colon X \to \R^n$ be a $\C^{\infty}$ map transverse to $0
\in \R^n$ and let $M$ be the union of some connected components of
$\varphi^{-1}(0)$. Then we obtain a framed submanifold $(M, F(\varphi))$
of $X$, where $F(\varphi) = (v_1, \ldots, v_n)$ and $(v_1(x), \ldots,
v_n(x))$ is transformed onto the canonical basis for $\R^n$ by the
isomorphism $N_xM \to \R^n$ induced by the derivative $d_x\varphi \colon
T_xX \to T_0\R^n = \R^n$ for every $x \in M$.

\begin{proof}[Proof of Theorem~\ref{th-1-6}]
The case $\dim V < n$ is obvious since then $f$ is null homotopic.
Suppose that $\dim V \geq n$, and let $(M, F)$ be a framed submanifold
of $V$ which represents $\PT(f)$. By a standard transversality argument,
we obtain a $\C^{\infty}$ map $\varphi \colon V \to \R^n$ such that
$\varphi$ is transverse to $0 \in \R^n$, $M$ is the union of some
connected components of $\varphi^{-1}(0)$, and $(M,F)$ is framed
cobordant to $(M, F(\varphi))$. In view of the Weierstrass approximation
theorem, there exists a regular map $\psi \colon V \to \R^n$ arbitrarily
close to $\varphi$ in the $\C^{\infty}$ topology. Then $\psi$ is
transverse to $0 \in \R^n$ and $Z \coloneqq \psi^{-1}(0)$ is isotopic to
$\varphi^{-1}(0)$, cf. \cite[p.~51]{bib1}. In particular, $(M,
F(\varphi))$ is framed cobordant to $(N, F(\psi))$, where $N$ is the
union of suitable connected components of $Z$. Furthermore, $Z \subseteq
V$ is a nonsingular Zariski closed subvariety. Set $K \coloneqq Z
\setminus N$, and let $T$, $\alpha$ be as in Lemma~\ref{lem-7-1} for
some nonnegative integer $l$. According to the Łojasiewicz inequality
\cite[Corollary~2.6.7]{bib6}, we can find an open neighborhood $U
\subseteq V$ of $K$, a real number $c > 0$, and an integer $q > 0$ such
that
\begin{equation*}
\norm{\psi(x)} \geq c\alpha(x)^{2q}\ \textrm{for all} \ x \in U,
\end{equation*}
where $\norm{-}$ stands for the Euclidean norm on $\R^n$.

Set $a \coloneqq (0, \ldots, 0, 1) \in \SB^n$, $b \coloneqq -a$, and let $\rho
\colon \SB^n \setminus \{a\} \to \R$ be the stereographic projection.
Fix a nonnegative integer $k$. If $l$ and $r$ are sufficiently large
integers, then the map $g \colon V \to \SB^n$ defined by
\begin{equation*}
g(x) =
\begin{cases}
\rho^{-1}(\frac{\psi(x)}{\alpha(x)^{2(q+r)}}) & \textrm{for} \ x \in V
\setminus K\\
a & \textrm{for} \ x \in K
\end{cases}
\end{equation*}
is of class $\C^k$. Since $\rho$ is a biregular isomorphism, the
restrictions $\restr{g}{T \setminus K}$ and $\restr{g}{\Int T}$ are
regular maps. Consequently, $\SC \coloneqq \{Z, \partial T, V \setminus
(Z \cup \partial T)\}$ is a stratification of $V$, and the map $g$ is
piecewise $\SC$-regular.

It remains to prove that $g$ is homotopic to $f$ or, equivalently,
$\PT(g) = \PT(f)$. To this end, let $B$ be the basis for $T_b\SB^n$
which corresponds to the canonical basis for $\R^n$ via the isomorphism
$d_b \rho \colon T_b\SB^n \to \R^n$. Note that the restriction of $g$ to
$g^{-1}(\SB^n \setminus \{a\}) = V \setminus K$ is a $\C^{\infty}$ map
transverse to $b \in \SB^n$ and $g^{-1}(b) = N$. It readily follows that
$(N, F(g,B))$ is framed cobordant to $(N, F(\varphi))$. Hence $(N,
F(g,B))$ is framed cobordant to $(M,F)$, which implies that $\PT(g)
= \PT(f)$, as required.
\end{proof}

\begin{remark}\label{rem-7-2}
As demonstrated above, if $\dim V \geq n$, then the stratification $\SC$
that appears in Theorem~\ref{th-1-6} can be chosen of the form
\begin{equation*}
\SC = \{Z, W, V \setminus (Z \cup W)\},
\end{equation*}
where $Z$ and $W$ are disjoint Zariski closed subvarieties of $V$ with
$\codim_V Z = n$ and $\codim_V W = 1$.
\end{remark}

\phantomsection


\addcontentsline{toc}{section}{\refname}

\begin{thebibliography}{99}

\begin{samepage}
\bibitem{bib1} R.~Abraham and J.~Robbin, Transversal Mappings and Flows,
Benjamin, 1967.\nopagebreak

\bibitem{bib2} R.~Benedetti and A.~Tognoli, On real algebraic vector
bundles, Bull. Sci. Math. 104 (1980), no. 2, 89--112.
\end{samepage}

\bibitem{bib3} R.~Benedetti and A.~Tognoli, Remarks and counterexamples
in the theory of real vector bundles and cycles, in: G\'eom\'etrie
alg\'ebrique r\'eelle et formes quadratiques, Lecture Notes in Math. 959
(1982), Springer, 198--211.

\bibitem{bib4} M.~Bilski, W.~Kucharz, A.~Valette and G.~Valette,
Vector bundles and regulous maps, Math. Z. 275 (2013), 403--418.

\bibitem{bib5} J.~Bochnak, M.~Buchner and W.~Kucharz, Vector bundles
over real algebraic varieties, K-Theory 3 (1989), 271--298. Erratum in
K-Theory 4 (1990), 113.

\bibitem{bib6} J.~Bochnak, M.~Coste and M.-F.~Roy, Real Algebraic
Geometry, Ergeb. Math. Grenzgeb. Vol.~36, Springer, 1989.

\bibitem{bib7} J.~Bochnak and W.~Kucharz, Algebraic approximation of
mappings into spheres, Michigan Math. J. 34 (1987), 119--125.

\bibitem{bib8} J.~Bochnak and W.~Kucharz, Realization of homotopy
classes by algebraic mappings, J.~Reine Angew. Math. 377 (1987),
159--169.

\bibitem{bib9} J.~Bochnak and W.~Kucharz, On real algebraic morphisms
into even-dimensional spheres, Ann. of Math. 128 (1988), 415--433.

\bibitem{bib10} J.~Bochnak and W.~Kucharz, Algebraic models of smooth
manifolds, Invent. Math. 97 (1989), 585--611.

\bibitem{bib11} J.~Bochnak and W.~Kucharz, K-Theory of real algebraic
surfaces and threefolds, Math. Proc. Cambridge Philos. Soc. 106 (1989),
471--480.

\bibitem{bib12} J.~Bochnak and W.~Kucharz, On vector bunles and real
algebraic morphisms, in: Real Analytic and Algebraic Geometry, Lecture
Notes in Math. 1420 (1990), Springer, 65--71.

\bibitem{bib13} J.~Bochnak and W.~Kucharz, Polynomial mappings from
products of algebraic sets into spheres, J.~Reine Angew. Math. 417
(1991), 135--139.

\bibitem{bib14} J.~Bochnak and W.~Kucharz, Vector bundles on a product
of real cubic curves, K-Theory 6 (1992), 487-497.

\bibitem{bib15} J.~Bochnak and W.~Kucharz, Algebraic cycles and
approximation theorems in real algebraic geometry, Trans. Amer. Math.
Soc. 337 (1993), 463--472.

\bibitem{bib16} J.~Bochnak and W.~Kucharz, The homotopy groups of some
spaces of real algebraic morphisms, Bull. London Math. Soc. 25 (1993),
385--392.

\bibitem{bib17} J.~Bochnak and W.~Kucharz, Elliptic curves and real
algebraic morphisms, J.~Algebraic Geom. 2 (1993), 635--666.

\bibitem{bib18} J.~Bochnak and W.~Kucharz, The Weierstrass approximation
theorem and a characterization of the unit circle, Proc. Amer. Math.
Soc. 127 (1999), 1571--1574.

\bibitem{bib19} J.~Bochnak and W.~Kucharz, Smooth maps and real
algebraic morphisms, Canad. Math. Bull. 42 (1999), 445--451.

\bibitem{bib20} J.~Bochnak and W.~Kucharz, The Weierstrass approximation
theorem for maps between real algebraic varieties, Math. Ann. 314
(1999), 601--612.

\bibitem{bib21} J.~Bochnak and W.~Kucharz, Line bundles, regular
mappings and the underlying real algebraic structure of complex
algebraic varieties, Math. Ann. 316 (2000), 793--817.

\bibitem{bib22} J.~Bochnak and W.~Kucharz, Vector bundles on a product
of real algebraic curves, Proc. Amer. Math. Soc. 133 (2005), 1617--1620.

\bibitem{bib23} J.~Bochnak and W.~Kucharz, Real algebraic morphisms
represent few homotopy classes, Math. Ann. 337 (2007), 909--921.

\bibitem{bib24} J.~Bochnak and W.~Kucharz, Algebraic approximation of
smooth maps, Univ. Iagell. Acta Math. 48 (2010), 9--40.

\bibitem{bib25} J.~Bochnak, W.~Kucharz and R.~Silhol, Morphisms, line
bundles and moduli spaces in real algebraic geometry, Publ. Math. Inst.
Hautes \'Etudes Sci. 86 (1997), 5--65. Erratum in Publ. Math. Inst.
Hautes \'Etudes Sci. 92 (2000), 195.

\bibitem{bib26} G.E.~Bredon, Topology and Geometry, Springer, 1993.

\bibitem{bib27} E.G.~Evans, Projective modules as fiber bundles, Proc.
Amer. Math. Soc. 27 (1971), 623--626.

\bibitem{bib28} G.~Fichou, J.~Huisman, F.~Mangolte and J.-Ph.~Monnier,
Fonctions r\'egulues, J.~Reine Angew. Math. 718 (2016), 103--151.

\bibitem{bib29} G.~Fichou, J.-Ph.~Monnier and R.~Quarez, Continuous
functions in the plane regular after one blowing up, Math. Z. 285
(2017), 287--323.

\bibitem{bib30} O.~Forster, Lectures on Riemann Surfaces, Springer,
1981.

\bibitem{bib31} R.~Fossum, Vector bundles over spheres are algebraic,
Invent. Math. 8 (1969), 222--225.

\bibitem{bib32} R.~Ghiloni, On the space of morphisms into generic real
algebraic varieties, Ann. Sc. Norm. Super. Pisa Cl. Sci. (5) 5 (2006),
419--438.

\bibitem{bib33} R.~Ghiloni, Second order homological obstructions on
real algebraic manifolds, Topology Appl. 154 (2007), 3090--3094.

\bibitem{bib34} J.~Huisman, A~real algebraic bundle is strongly
algebraic whenever its total space is affine, Contemp. Math. 182 (1995),
117--119.

\bibitem{bib35} J.~Huisman, Correction to ``A~real algebraic bundle is strongly
algebraic whenever its total space is affine'', Contemp. Math. 253
(2000), 179.

\bibitem{bib36} D.~Husemoller, Fibre Bundles, 3rd edition, Springer,
1974.

\bibitem{bib37} N.V. Ivanov, Approximation of smooth manifolds by real
algebraic sets, Russian Math. Surveys 37:1 (1982), 1--59.

\bibitem{bib38} N.~Joglar-Prieto, Rational surfaces and regular maps
into the $2$-dimensional sphere, Math. Z. 234 (2000), 399--405.

\bibitem{bib39} N.~Joglar-Prieto and J.~Koll\'ar, Real abelian
varieties with many line bundles, Bull. London Math. Soc. 35 (2003),
79--84.

\bibitem{bib40} N.~Joglar-Prieto and F.~Mangolte, Real algebraic
morphisms and del Pezzo surfaces of degree~$2$, J.~Algebraic Geom. 13
(2004), 269--285.

\bibitem{bib41} J.~Koll\'ar, W.~Kucharz and K.~Kurdyka, Curve-rational
functions, Math. Ann., posted on 2017, DOI 10.1007/s00208-016-1513-z.

\bibitem{bib42} J.~Koll\'ar and K.~Nowak, Continuous rational functions
on real and $p$-adic varieties, Math. Z. 279 (2015), 85--97.

\bibitem{bib43} A.A.~Kosinski, Differential Manifolds, Dover, 1993.

\bibitem{bib44} W.~Kucharz, Vector bundles over real algebraic surfaces
and threefolds, Compos. Math. 60 (1986), 209--225.

\bibitem{bib45} W.~Kucharz, Algebraic cycles and vector bundles on real
affine threefolds, Manuscripta Math. 60 (1988), 211--216.

\bibitem{bib46} W.~Kucharz, How to make vector bundles algebraic, C.R.
Rep. Sci. Canada 11 (1989), 231--235.

\bibitem{bib47} W.~Kucharz, Rational maps in real algebraic geometry,
Adv. Geom. 9 (2009), 517--539.

\bibitem{bib48} W.~Kucharz, Complex cycles on algebraic models of
smooth manifolds, Math. Ann. 346 (2010), 829--856.

\bibitem{bib49} W.~Kucharz, Regular versus continuous rational maps,
Topology Appl. 160 (2013), 1375--1378.

\bibitem{bib50} W.~Kucharz, Regular maps into real Fermat varieties,
Bull. London Math. Soc. 451 (2013), 1086--1090.

\bibitem{bib51} W.~Kucharz, Continuous rational maps into the unit
$2$-sphere, Arch. Math. (Basel) 102 (2014), 257--261.

\bibitem{bib52} W.~Kucharz, Approximation by continuous rational maps
into spheres, J.~Eur. Math. Soc. 16 (2014), 1555--1569.

\bibitem{bib53} W.~Kucharz, Complex cycles as obstructions on real
algebraic varieties, Glasgow Math. J. 57 (2015), 343--347.

\bibitem{bib54} W.~Kucharz, Some conjectures on stratified-algebraic
vector bundles, J.~of Singul. 12 (2015), 92--104.

\bibitem{bib55} W.~Kucharz, Continuous rational maps into spheres, Math.
Z. 283 (2016), 1201--1215.

\bibitem{bib56} W.~Kucharz, Stratified-algebraic vector bundles of small
rank, Arch. Math. (Basel) 107 (2016), 239--249.

\bibitem{bib57} W.~Kucharz and K.~Kurdyka, Some conjectures on
continuous rational maps into spheres, Topology Appl. 208 (2016),
17--29.

\bibitem{bib58} W.~Kucharz and K.~Kurdyka, Stratified-algebraic vector
bundles, J.~Reine Angew. Math., posted on 2016, DOI
10.1515/crelle-2015-0105.

\bibitem{bib59} W.~Kucharz and K.~Kurdyka, Linear equations on real
algebraic surfaces, Manuscripta Math., posted on 2017, DOI
10.1007/s00229-017-0925-8.

\bibitem{bib60} W.~Kucharz and K.~Kurdyka, Comparison of
stratified-algebraic and topological K-theory, available at arXiv:
1511.04238 [math.AG].

\bibitem{bib61} W.~Kucharz and Ł.~Maciejewski, Complexification and
homotopy, Homology Homotopy Appl. 16 (2014), 159--165.

\bibitem{bib62} W.~Kucharz and K.~Rusek, An application of ample vector
bundles in real algebraic geometry, Proc. Amer. Math. Soc. 139 (2011),
1155--1161.

\bibitem{bib63} W.~Kucharz and M.~Zieliński, Reguluos vector bundles,
available at arXiv: 1703.05566 [math.AG].

\bibitem{bib64} J.M.~Lee, Introduction to Smooth Manifolds, Springer,
2003.

\bibitem{bib65} J.-L.~Loday, Applications alg\'ebriqes du tore dans la
sph\'ere et de $\SB^p \times \SB^q$ dans $\SB^{p+q}$, in: Algebraic
K-Theory II, Lecture Notes in Math. 342, Springer, 1973, 79--91.

\bibitem{bib66} K.~Lønsted, An algebraization of vector bundles on
compact manifolds, J.~Pure Appl. Algebra 2 (1972), 193--207.

\bibitem{bib67} K.~Lønsted, Vector bundles over finite CW complexes are
algebraic, Proc. Amer. Math. Soc. 38 (1973), 27--31.

\bibitem{bib68} F.~Mangolte, Real algebraic morphisms on $2$-dimensional
conic bundles, Adv. Geom. 6 (2006), 199--213.

\bibitem{bib69} M.G.~Marinari and M.~Raimondo, Fibrati vettoriali su
variet\`a algebriche definite su corpi non algebricamente chiusi, Boll.
Un. Mat. Ital. A 16 (1979), no. 5, 128--136.

\bibitem{bib70} J.W.~Milnor, Topology from the Differentiable Viewpoint,
The University Press of Virginia, Charlottesville, Va. 1965.

\bibitem{bib71} J.-P.~Monnier, Semi-algebraic geometry with rational
continuous functions, available at arXiv: 1603.04193 [math.AG].

\bibitem{bib72} K.J.~Nowak, Some results of algebraic geometry over
Henselian rank one valued fields, Sel. Math. New Ser. 28 (2017),
455--495.

\bibitem{bib73} Y.~Ozan, On entire rational maps in real algebraic
geometry, Michigan Math. J. 42 (1995), 141--145.

\bibitem{bib74} Y.~Ozan, On algebraic K-theory of real algebraic
varieties with circle action, J.~Pure Appl. Algebra 170 (2002),
287--293.

\bibitem{bib75} J.~Peng and Z.~Tang, Algebraic maps from spheres to
spheres, Sci. China Ser. A, 42 (1999), 1147--1154.

\bibitem{bib76} R.G.~Swan, Topological examples of projective modules,
Trans. Amer. Math. Soc. 230 (1977), 201--234.

\bibitem{bib77} R.G.~Swan, Vector bundles, projective modules and the
K-theory of spheres, in: Algebraic Topology and Algebraic K-Theory, Ann.
of Math. Stud. 113 (1987), Princeton University Press, 432--522.

\bibitem{bib78} A.~Tognoli, Algebraic Geometry and Nash Functions,
Institutiones Math. 3, New York, Academic Press, 1978.

\bibitem{bib79} F.-J.~Turiel, Polynomial maps and even-dimensional
spheres, Proc. Amer. Math. Soc. 135 (2007), 2665--2667.

\bibitem{bib80} R.~Wood, Polynomial maps from spheres to spheres,
Invent. Math. 5 (1968), 163--168.

\bibitem{bib81} M.~Zieliński, Homotopy properties of some real algebraic
maps, Homology Homotopy Appl. 18 (2016), 287--294.

\end{thebibliography}
\end{document}